\documentclass{amsart}
\usepackage[latin2]{inputenc}
\usepackage{amsmath,amssymb,amsbsy,amsfonts}
\usepackage{color}
\usepackage{epsfig} 
\usepackage{graphics} 
\usepackage{a4wide}
\usepackage[margin=2cm]{geometry}

\def\tens#1{\pmb{\mathsf{#1}}}
\def\vec#1{\boldsymbol{#1}}

\def\diver{\mathop{\mathrm{div}}\nolimits}

\def\tr{\mathop{\mathrm{tr}}\nolimits}

\def\dd#1#2{\frac{\d #1}{\d #2}}

\def\bef{\vec{f}}

\def\bn{\vec{n}}

\def\btau{\vec{\tau}}

\def\bH{\tens{H}}
\def\bG{\tens{G}}

\def\bT{\tens{T}}
\def\bD{\tens{D}}

\def\bS{\tens{S}}
\def\bI{\tens{I}}

\def\bv{\vec{v}}

\def\b0{\vec{0}}

\def\bw{\vec{w}}

\def\bu{\vec{u}}

\def\obS{\bS}

\def\Wn#1{W^{1,#1}_{\bn}}

\def\Lnd#1{L^{#1}_{\bn, \diver}}

\def\Wnd#1{W^{1,#1}_{\bn, \diver}}

\def\Wndm#1{W^{-1,#1}_{\bn, \diver}}

\def\Wnm#1{W^{-1,#1}_{\bn}}

\renewcommand{\div}{{\rm div}}
\def\bff{\vec{f}}
\def\bxi{\vec{\xi}}

\def\bG{\tens{G}}

\newcommand{\R}{\mathbb{R}}

\def\XXint#1#2#3{{\setbox0=\hbox{$#1{#2#3}{\int}$ }
\vcenter{\hbox{$#2#3$ }}\kern-.6\wd0}}

\def\dd{\,{\rm d}}

\newtheorem{Theorem}{Theorem}[section]

\newtheorem{Lemma}{Lemma}[section]

\newtheorem*{Remarkno}{Remark}

\title{Analysis of a viscosity  model for concentrated polymers}
\author[M. Bul\'i\v{c}ek]{Miroslav Bul\'i\v{c}ek}
\address{Miroslav Bul\'i\v{c}ek. Mathematical Institute of Charles University,
Sokolovsk\'{a} 83, 186 75 Prague, Czech Republic}
\email{mbul8060@karlin.mff.cuni.cz}
\author[P.~Gwiazda]{Piotr Gwiazda}
\address{Piotr Gwiazda. University of Warsaw, Institute of Applied Mathematics and Mechanics,
Banacha~2, 02-097 Warsaw, Poland}
\email{pgwiazda@mimuw.edu.pl}
\author[E. S\"{u}li]{Endre S\"{u}li}
\address{Endre S\"{u}li. Mathematical Institute, University of Oxford, Woodstock Road, Oxford OX2 6GG, United Kingdom}
\email{endre.suli@maths.ox.ac.uk}
\author[A.~\'{S}wierczewska-Gwiazda]{Agnieszka \'{S}wierczewska-Gwiazda}
\address{Agnieszka \'{S}wierczewska-Gwiazda. University of Warsaw, Institute of Applied Mathematics and Mechanics,
Banacha~2, 02-097 Warsaw, Poland}
\email{aswiercz@mimuw.edu.pl}
\thanks{M.~Bul\'{\i}\v{c}ek thanks the project MORE (ERC-CZ project LL1202 financed by the Ministry of Education, Youth and Sports, Czech Republic) and the Neuron Fund for Support of Science. M.~Bul\'{\i}\v{c}ek is a member of the Ne\v{c}as Center for Mathematical Modeling. A.~\'{S}wierczewska-Gwiazda acknowledges the support of the project IdP2011/00066.  The research was partially supported by the Warsaw Center of Mathematics and Computer Science.  P.~Gwiazda is a coordinator of the International Ph.D. Projects Programme of the  Foundation for Polish Science operated within the Innovative Economy Operational Programme 2007--2013 (Ph.D. Programme: Mathematical Methods in Natural Sciences). E.~S\"{u}li acknowledges the financial support of the project MORE (ERC-CZ project LL1202 financed by the Ministry of Education, Youth and Sports, Czech Republic) and the Oxford Centre for Nonlinear Partial Differential Equations (OxPDE)}

\begin{document}


\begin{abstract}
The paper is concerned with a class of mathematical models for polymeric fluids, which involves the coupling of the Navier--Stokes equations for a viscous, incompressible, constant-density fluid with a parabolic-hyperbolic integro-differential equation describing the evolution of the polymer distribution function in the solvent, and a parabolic integro-differential equation for the evolution of the monomer density function in the solvent. The viscosity coefficient, appearing in the balance of linear momentum equation in the Navier--Stokes system, includes dependence on the shear-rate as well as on the weight-averaged polymer chain length. The system of partial differential equations under consideration captures the impact of polymerization and depolymerization effects on the viscosity of the fluid. We prove the existence of global-in-time, large-data weak solutions under fairly general hypotheses.
\end{abstract}

\maketitle

\section{Viscosity models for polymers -- formulation of the problem}
Contemporary approaches to the modelling of polymeric fluids have exploited multi-scale descriptions in an essential way. Mathematical models have thus been built by coupling systems of partial differential equations describing the motion of the solvent with equations that track the evolution of the microscopic behaviour of the solute in the solution. In this paper we focus on the modelling and the analysis of two phenomena: we wish to explore how the rheological properties of the fluid are affected by the presence of, possibly very long, chains of macromolecules; and, second, we aim to investigate the possible mechanisms for polymerization and depolymerization effects that can also heavily depend on the properties of the flow. The solvent is considered to be an incompressible fluid and we assume that the possible changes of the density are negligible relative to other phenomena. Thus, the underlying system of equations under consideration consists of the balance of linear momentum and the incompressibility constraint, i.e.,
\begin{equation}\label{balance-lm}
\begin{split}
\partial_t (\varrho \bv) + \diver_x (\varrho \bv \otimes \bv) - \diver_x \bT&= \varrho \bef,\\
\diver_x \bv&=0.
\end{split}
\end{equation}
These equations are assumed to be satisfied in a time-space cylinder $Q:=(0,T)\times \Omega$, with $\Omega \subset \mathbb{R}^d$, $d \in \{2,3\}$, being the physical flow domain. Here, $\bv:Q\to \mathbb{R}^d$ denotes the velocity of the solvent, $\varrho$ is the constant density, $\bef:Q\to \mathbb{R}^d$ represents the density of volume forces and $\bT:Q\to \mathbb{R}^{d\times d}$ is the Cauchy stress. In order to close the system \eqref{balance-lm}, one also needs to relate the Cauchy stress to other quantities describing the qualitative behaviour of the material. Despite their importance, in the present paper we shall, for the sake of simplicity, neglect all thermal effects and will focus instead on mechanical properties of the fluid in the isothermal setting. In particular, we are interested in models that link the Cauchy stress to the shear rate and to the contribution of the microscopic polymer molecules, in order to investigate how the level of interaction between polymer chains, their lengths,
 and their configuration influence the Cauchy stress.

That the presence of macromolecules in a solvent dramatically changes the properties of the flow was already observed in \cite{Fl40}, and has been, ever since, the focus of mathematical models. In general terms, one can find two different approaches to the problem: the empirical (rheological) one, where one typically fits the observed data in the model and introduces empirical formulae for the viscosity, $\nu$, based on the ratio of the magnitude of the Cauchy stress and the magnitude of the shear rate; and a second one, based on modelling the microstructure of the polymer under consideration using tools from statistical mechanics. The second approach has been particularly common in the case of dilute polymers, i.e., when the concentration of polymer molecules in the solvent is very low, so that the interactions of the macromolecules can be neglected. In such models the presence of macromolecules in the solvent is typically modelled by additively supplementing the Cauchy stress
tensor with an elastic stress tensor. This then leads to a closed system of partial differential equations; for a derivation of the model and further related mathematical results and references we refer the reader to Barrett \& S\"uli \cite{BS1,BS2,BS3,BS4,BS5} and Bul\'{\i}\v{c}ek, M\'alek \& S\"uli \cite{BMS}; see also \cite{OtTz2008} and \cite{BaTr2013}.

On the other hand, in the case of concentrated polymers one has to take into account the interactions of the polymer molecules (see, for example, \cite{DoiEdwardsbook}). Here we shall be concerned with models in which the only important property of the microstructure that influences the macroscopic flow equations is the polymer chain length, in the sense that the viscosity coefficient appearing in the balance of linear momentum equation is considered to be a function of the polymer chain length, resulting in a family of {\it viscosity models}.  There is an evident contrast between mathematical studies of models that ignore the dependence of the viscosity coefficient on the lengths of polymer chains, and the vast set of experimental data indicating the influence of the polymer chain length on the viscosity, see~\cite{BeFo68, NiMa98, Fl40, Te2002, YoBoBaGaUs2009} among others. In particular, none of the approaches mentioned above considered the impact of polymerization and depolymerization effects on the flow properties.
The goal of  this paper is therefore to explore a class of viscosity models that incorporate the influence of polymerization and depolymerization on the macroscopic flow variables. Specifically, we provide a rigorous proof of the existence of a global in time, large-data weak solution to the class of models under consideration.

\subsection{Viscosity models: the relationship between the viscosity, the shear rate and the polymer chain length}
It was experimentally observed more than seventy years ago already (cf. \cite{Fl40}, for example,) that the properties of the solvent heavily, and nonlinearly, depend on both the number of polymer chains in the solvent and the polymer chain lengths. Flory had studied the relationship between the viscosity of the solvent and the number of polymer chains in the solvent for linear polyesters at constant temperature, pressure and shear-rate. Guided by experimental evidence he proposed that for linear polyesters the logarithm of the viscosity should be a linear function of the square-root of the so-called weight-averaged chain length. Moreover, it was also observed that this relationship is independent of the type of distribution function for the species in the polymer. To be more concrete, if we denote by $\psi(r)$ the distribution function of polymer chains of length $r$ (meaning that the polymer chain consists of $r$ monomers) and denoting the minimal length of a polymer by
  $r_0$,
 the total weight of the polymer can be considered to be proportional to
$$
\int_{r_0}^{\infty} r \psi(r)\dd r.
$$
This is indeed the case if one considers $r_0 \gg 1$, but must be corrected for lower values of $r_0$, see \cite{Fl40}. Thus, the weight-averaged chain length can be defined as
$$
\tilde{\psi}:= \frac{\int_{r_0}^{\infty}r^2\psi(r)\dd r}{\int_{r_0}^{\infty} r\psi(r)\dd r}.
$$
It was empirically shown in \cite{Fl40} that
$$
\ln \nu \sim \sqrt{\tilde{\psi}},
$$
in the case when one considers linear polyesters; see \cite{Florybook} for a more detailed account. These experiments indicate that, generally, one does not need to consider specific values of the distribution function but that its average in the above sense will suffice from the point of view of accurate macroscopic modelling. This was also conjectured to be the case in many other situations, and quite recently also in the case of amines at high pressures, see \cite{YoBoBaGaUs2009}. In that paper, the authors showed that the logarithm of the viscosity is not necessarily related to the square-root of the averaged chain length but it can be a more general function, though still dependent only on the averaged properties and not on the fine details of $\psi(r)$.

Further, it was common belief (based on laboratory experiments at nonconstant shear rates) that many such phenomena can be explained by using a shear-dependent viscosity instead of a polymer-length-dependent viscosity. However, this is not the case for concentrated polymers. For example, in \cite{NiMa98}, the authors studied the simultaneous influence of the shear-rate and the polymer-length on the viscosity for polydisperse polymer melts. It was observed that, up to a certain critical value of the shear rate, the fluid behaves as a Newtonian fluid with viscosity depending only on the chain length, while for large shear-rates significant shear-thinning was observed. Moreover, it was also shown that, asymptotically, the viscosity depends only on the shear-rate and the influence of the chain length in large shear-rate regions can be neglected. In addition, it also transpired that the weight-averaged chain length was not an appropriate quantity in the cases considered, and the authors suggested instead the following general form of weighting:
\begin{equation}\label{chain_av}
\tilde{\psi}:= \frac{\int_{r_0}^{\infty}\omega(r)r\psi(r)\dd r}{\int_{r_0}^{\infty} r\psi(r)\dd r},
\end{equation}
where the function $\omega$ represents the most significant lengths of the polymer chains. Surprisingly, during the last decade, it was observed that there exist materials (e.g. polystyrene-decalin, polyethylen) that exhibit the opposite phenomenon: shear-thickening, see \cite{sherThic}. For all these reasons, we consider a much more general form of the Cauchy stress:
$$
\bT:=\bS - q\,\bI,
$$
where $q$ is the pressure and $\bS$ represents the viscous part of the Cauchy stress; we further assume that $\bS$ is of the form
\begin{equation} \label{T00}
{\bS(\tilde{\psi},\bD_x\bv )}:= \nu(\tilde\psi, |\bD_x\bv|)\,\bD_x\bv,
\end{equation}
where $\bD_x \bv$ is the symmetric velocity gradient and $\tilde{\psi}$ is a general weight-averaged chain length function of the form \eqref{chain_av}, which one can simplify (without loss of generality) to the form
$$
\tilde{\psi}:= \int_{r_0}^{\infty} \gamma(r)\psi (r)\dd r.
$$
We shall not specify the particular form of $\nu$, but will instead admit a general class of viscosities so as to
enable the consideration of both shear-thinning and shear-thickening fluids.

\subsection{Polymerization models}
In the above system the microscopic parameter that is taken into account is the polymer chain length, and we describe the process of elongation of these chains (by binding free monomers to polymer chains) as well as the process of breaking longer chains into shorter ones. We shall view this microscopic structure as an evolution of two populations: the population of polymer chains and the population of monomers.
This kind of description is akin to the model
for prion dynamics considered in
\cite{GrPuWe2006}; see also
\cite{CaLeDoDeMoPe2009, CaLeOeDePe2009}. As the model formulated in the current paper can be seen as an extension of the prion proliferation model in those papers to the case with spatial effects (transport by a solvent and spatial diffusion) we shall discuss it in more detail. The authors of \cite{GrPuWe2006,CaLeDoDeMoPe2009, CaLeOeDePe2009} study how the healthy prion protein and infectious prion protein populations interact in an infected organism.  The infectious prions are abnormal pathogenic conformations of the normal ones. The proposed model  considers  the infectious prion proteins to be a polymeric form of normal prion proteins.  Polymers of
infected prions can split into shorter chains. Such a splitting  transforms one infectious
polymer into two shorter infectious polymers,  which
can then attach again to normal prion proteins. However, when the length of a part of a split polymer falls
below a certain critical value, it immediately degrades into a normal prion protein, namely a monomer.
By $\psi$ we denote the distribution function of polymer chains of length $r>r_0$  that satisfy the following equation:
\begin{equation}\label{eq:psi_MB}
\begin{split}
\partial_t \psi&(t,r)  +\tau \phi(t)\partial_r\psi(t,r)
= -\beta(r)\psi(t,r) +2\int_r^{\infty}\beta(\tilde r)\kappa(r,\tilde r)\psi(t,\tilde r)\dd\tilde r.
\end{split}
\end{equation}
The term
  $\tau \phi(t)\partial_r\psi(t,r)$ represents
the gain in length of
polymer chains due to polymerization with rate $\tau>0$,
 $\beta(r)$ is the fragmentation rate, namely
the
length-dependent likelihood of splitting of polymers to
monomers,
 $\kappa(r,\tilde r)$ is the probability that a polymer chain will split into two shorter polymer chains of length $r$ and $\tilde r-r$, respectively, the term  $-\beta(r)\psi(t,r)$ is
the loss of polymer chains subject to the splitting rate $\beta(r)$, and
  the last term  is the count of all the polymer chains of length $r$ resulting from the splitting of
polymer chains of length greater than~$r$.

The evolution of  monomers is described by means of
the function $\phi(t)$, which  is the concentration of free monomers at time $t$, satisfying  the equation
\begin{equation}\label{eq:phi_MB}
\begin{split}
\frac{\dd}{\dd t}\phi&(t)= 2\int_{0}^{r_0}r\int_r^{\infty}\beta(\tilde r)\kappa(r,\tilde r)\psi(t,\tilde r)\dd\tilde r \dd r
-\phi(t)\int_{r_0}^\infty\tau\psi(t,r)\dd r.
\end{split}
\end{equation}
The first term on the right-hand side  represents the monomers gained when a
 polymer chain splits with at least one polymer chain shorter than
the minimum length $r_0$, while the second  term
 is the loss of monomers as they are polymerized.





\subsection{The complete model}
We conclude this introductory section with the precise statement of the complete system of equations under consideration. For  a given Lipschitz domain $\Omega \subset \mathbb{R}^d$, $d \in \{2,3\}$, and a given final time $T>0$, we consider the balance of linear momentum and the incompressibility constraint
 on $Q:=(0,T)\times \Omega$ in the form (after scaling by the constant density $\varrho$)
\begin{equation}\label{NS}
\begin{split}
\partial_t\bv(t,x)+ \div_x (\bv(t,x)\otimes \bv(t,x)) +\nabla_x q(t,x)- {\rm div}_x \bS(\tilde\psi(t,x),\bD_x \bv(t,x))&= \bef,\\
{\rm div}_x \bv(t,x) &= 0,
\end{split}
\end{equation}
where $\bv:Q \to \mathbb{R}^d$ denotes the velocity of the fluid (solvent), $q:Q\to \mathbb{R}$ is the pressure, and $\bef:Q\to \mathbb{R}^d$ is the density of the external body forces. The  viscous  part of the Cauchy stress $\bS:Q\to \mathbb{R}^{d\times d}$ is given by the formula
\begin{equation} \label{T}
{\bS(\tilde{\psi}(t,x),\bD_x \bv(t,x))}:= \nu(\tilde\psi(t,x), |\bD_x\bv(t,x)|)\,\bD_x\bv(t,x).
\end{equation}
Here, $\bD_x \bv$ denotes the symmetric velocity gradient, i.e., $\bD_x \bv:=\frac12 (\nabla_x \bv + (\nabla_x \bv)^{\rm T})$, and the viscosity coefficient $\nu:\R_+\times\R_+\to\R_+$ is allowed to depend on the shear-rate $|\bD_x\bv|$ and on the averaged polymer distribution function $\tilde\psi :Q\to \mathbb{R}_+$, defined by
\begin{equation}\label{psi-aver}
\tilde{\psi}(t,x):=\int_{\R_0} \gamma(r)\psi(t,x,r)\dd r,
\end{equation}
where $\R_0:=(r_0,\infty)$ and $\gamma:\R_0 \to \mathbb{R}_+$ is a continuous nonnegative function, representing the weight function associated with the averaging of polymer chain lengths. The distribution function
$\psi:Q\times \R_0\to\R_+$, which in the absence of fluid motion satisfies \eqref{eq:psi_MB}, is now assumed to satisfy the following equation:
\begin{equation}\label{eq:psi}
\begin{split}
\partial_t \psi&(t,x,r) + \bv(t,x) \cdot \nabla_x\psi(t,x,r) +\tau(r) \phi(t,x)\partial_r\psi(t,x,r)-A(r)\Delta_x\psi(t,x,r)
\\=& -\beta(r, \bv,\bD_x\bv)\psi(t,x,r) +2\int_r^{\infty}\beta(\tilde r,\bv, \bD_x\bv)\kappa(r,\tilde r)\psi(t,x,\tilde r) \dd\tilde r,
\end{split}
\end{equation}
in $Q\times\R_0$. There are therefore two additional terms compared with \eqref{eq:psi_MB}: the convective/transport term $\bv \cdot \nabla_x\psi$ due to the motion of the solvent, and the diffusion term $A\Delta_x\psi$ associated with the Brownian force acting on the polymer molecules immersed in the solvent. The parameter $r_0\in [0,\infty)$ is a fixed minimal length of a polymer molecule, $A(r):\R_0\to \mathbb{R}_+$ denotes the rate of diffusion for a particular value of $r$ and is assumed here to be a nonincreasing function of $r$,
$\tau:[r_0,\infty)\to \mathbb{R}_+$ is the polymerization rate, $\beta:\R_0\times \mathbb{R}^d\times \mathbb{R}^{d\times d} \to \mathbb{R}_+$ is the fragmentation rate of polymer chains, which, unlike \eqref{eq:psi_MB}, will also be allowed to depend on the macroscopic quantities in the model, namely on the fluid velocity and the shear rate, and, finally,
$\kappa(r,\tilde r)$ denotes the probability that a polymer molecule of length $\tilde{r}$ will split into two polymer molecules of lengths $r$ and $\tilde r-r$, respectively.
The function $\phi$ appearing in \eqref{eq:psi} is the concentration of free monomers satisfying, in the absence of fluid motion, the identity \eqref{eq:phi_MB}, and which, in the presence of a moving solvent, is assumed to satisfy the following equation:
\begin{equation}\label{eq:phi}
\begin{split}
&\partial_t \phi(t,x)+\bv(t,x) \cdot \nabla_x\phi(t,x)
-A_0\Delta_x\phi(t,x)
\\
&\quad =-\phi(t,x)\int_{r_0}^\infty\partial_r (r\tau(r))\psi(t,x,r) \dd r + 2 \int_0^{r_0} r \int_{r_0}^{\infty}\beta(\tilde r,\bv, \bD_x\bv)\kappa(r,\tilde r)\psi(t,x,\tilde r) \dd\tilde r \dd r
\end{split}
\end{equation}
in the space time cylinder $Q$. Here, the rate of diffusion $A_0$ (caused by Brownian noise) is assumed to be a positive constant and the additional transport term
$\bv \cdot \nabla_x\phi$ is due to the flow of the solvent.

We shall further assume that the velocity satisfies a Navier slip boundary condition, i.e.,
\begin{eqnarray}
\begin{aligned}
\bv \cdot \bn &= 0 &\textrm{ on }\partial \Omega,\\
(\bS \bn)_{\btau}&=-\alpha^* \bv &\textrm{ on } \partial \Omega,
\end{aligned}
\label{BC-v}
\end{eqnarray}
where $\alpha^*\ge 0$ is the domain-wall friction coefficient, $\bn$ is the unit outward normal vector to $\partial\Omega$, and for any $\bv$ we have denoted by $\bv_{\btau}:= \bv-(\bv\cdot \bn)\bn$ the projection of $\bv$ on the tangent hyperplane to the boundary. The function $\phi$ is assumed to satisfy a homogeneous Neumann boundary condition with respect to the
$x$ variable, i.e.,
\begin{equation}
\begin{split}\label{BC-phi}
\nabla_x \phi \cdot \bn =0 \textrm{ on } \partial \Omega,
\end{split}
\end{equation}
and for $\psi$ we also prescribe a homogeneous Neumann boundary condition with respect to the $x$ variable, i.e.,
\begin{equation}
\begin{split}\label{BC-psi}
\nabla_x \psi \cdot \bn =0 \textrm{ on } \partial \Omega.
\end{split}
\end{equation}
We shall further assume that $\psi$ vanishes at infinity with respect to $r$, i.e.,
\begin{equation}
\begin{split}\label{BCr-psi}
\lim_{r\to \infty}\psi(t,x,r)=0.
\end{split}
\end{equation}
Finally, to complete the statement of the problem, we prescribe the following set of initial conditions: for the velocity field we assume that
\begin{equation}
\bv(0,x)=\bv_0(x) \textrm{ in } \Omega, \qquad \diver_x \bv_0=0 \textrm{ in } \Omega, \qquad \bv_0\cdot \bn = 0 \textrm{ on }\partial \Omega,\label{IC-v}
\end{equation}
and for $\phi$ and $\psi$ we assume that
\begin{alignat}{4}
\phi(0,x)&=\phi_0(x) &&\textrm{ in } \Omega, &&\quad \phi_0\ge 0,\label{IC-phi}\\
\psi(0,x,r)&=\psi_0(x,r) &&\textrm{ in } \Omega\times (r_0,\infty), &&\quad \psi_0\ge 0.\label{IC-psi}
\end{alignat}

\subsection{Assumptions on the parameters}
In what follows $K$ will signify a universal  positive constant. We adopt the following assumptions on the data:
\begin{itemize}
\item[(A1)] The diffusion rate $A:\R_0\to \mathbb{R}_+$ is a continuous nonincreasing strictly positive function defined on $\R_0$ such that $\lim_{r \rightarrow \infty} A(r)=0$.
\item[(A2)] The polymerization rate $\tau$ is a nondecreasing bounded globally Lipschitz continuous function such that $\tau(r_0)=0$, $\tau'(r_0)>0$, and there
exists a positive constant $K$ such that
\begin{equation}
\label{tau-a}
 K^{-1}r_0\le \tau(r) + r\,\tau'(r)\quad\mbox{and} \quad
\tau(r) + \tau'(r) + r\,\tau'(r)+\frac{\tau(r)}{r}\le K.
\end{equation}
\item[(A3)] The fragmentation  rate $\beta:\R_0\times \mathbb{R}^d\times \mathbb{R}^{d\times d}\to \mathbb{R}_+$ is a smooth bounded function, which is increasing with respect to its first variable, and satisfies, for all $(r,\bv,\bD)\in \R_0\times \R^d\times\mathbb{R}^{d\times d}$,
    \begin{equation}
0<\beta(r,\bv,\bD) \le K.\label{beta-a}
    \end{equation}
Moreover, we require that the function $\eta$, defined by
\begin{equation}\label{dfeta}
\eta(r):=\sup_{\bu \in \mathbb{R}^d, \, \bD\in \mathbb{R}^{d\times d}} \frac{\partial_r\beta(r,\bu,\bD)}{\beta(r,\bu,\bD)},
\end{equation}
is measurable and nonnegative, and
\begin{equation}\label{eta-as}
0\le (1+r)\eta(r)\le K, \qquad \int_{r_0}^{\infty}\eta(r)\dd r \le K.
\end{equation}f
\item[(A4)] The kernel function $\kappa$ is, for simplicity, defined as
\begin{equation}
\label{df-kappa}
\kappa(r,\tilde{r}):=\left\{\begin{aligned}&\frac{1}{\tilde{r}} &&\textrm{if } \tilde{r}>r_0 \textrm{ and } 0<r<\tilde{r},\\
&0 &&\textrm{otherwise}.
\end{aligned}\right.
\end{equation}
It therefore follows that, for all $\alpha>0$,
\begin{equation}\label{kappa}
 \int_0^{\infty}r^{\alpha-1}\kappa(r,\tilde r)\dd r=\frac{\tilde{r}^{\alpha-1}}{\alpha}
 \qquad \mbox{for $\tilde{r}>r_0$}.
 \end{equation}
\item[(A5)] There exists a positive real number $\theta>0$ such that, for all $r\in [r_0,\infty)$, one has
\begin{equation}\label{growthgamma}
\gamma(r)\le K(1+ r)^{\theta}.
\end{equation}
\item[(A6)] We assume that $\nu:\R_+\times\R_+\to\R_+$ is a continuous function such that, for some $p>\frac{2d}{d+2}$ and for all $\bxi, \tilde{\bxi}\in\R^{d\times d}$ fulfilling $\bxi\neq \tilde{\bxi}$, one has (cf. \eqref{T00})
\begin{equation} \label{S}
\begin{split}
|{\bS(\cdot,\bxi)}|&\le K(1+|\bxi|)^{p-1},\\
{\bS(\cdot,\bxi)}\cdot\bxi&\ge K^{-1}|\bxi|^p-K,\\
(\bS(\cdot, \bxi)-\bS(\cdot,\tilde{\bxi}))\cdot (\bxi - \tilde{\bxi})&> 0.
\end{split}
\end{equation}
\end{itemize}
\subsection{The main result}
Before introducing the definition of \textit{weak  solution} to problem \eqref{NS}--\eqref{IC-psi} we shall summarize
our notational conventions. We denote by $T\in (0, \infty)$ the length of the time  interval and by $\Omega \subset \mathbb{R}^d$, $d \in \{2,3\}$, a bounded domain in $\mathbb{R}^d$ with $\mathcal{C}^{1,1}$-boundary $\partial \Omega$; we then write $\Omega \in \mathcal{C}^{1,1}$. We also set $Q:=(0,T)\times \Omega$ and $\Gamma:=(0,T)\times\partial\Omega$.

For $p\in [1, \infty]$ we define the Lebesgue space $L^p(\Omega)$ and the Sobolev spaces $W^{1,p}(\Omega)$ in the
usual way, and we denote the trace on $\partial\Omega$ of a Sobolev function $u$ by $\tr u$. If $X$ is a Banach
space, then $X^d:=X\times \dots\times X$ and we write $X^*$ for the
dual space of $X$; $L^p(0,T; X)$ denotes the Bochner space of $X$-valued $L^p$ functions defined on $(0,T)$. For (scalar, vector- or tensor-valued) functions $f$
and $g$ and $\cdot$ signifying the product of real numbers, scalar product of vectors, or scalar product of tensors, as the case may be, we shall write
\begin{align*}
\begin{aligned}
(f,g)&:= \int_{\Omega}f (x) \cdot g(x) \dd x \quad &&\textrm{ if } f\cdot g\in L^1(\Omega),\\
(f,g)_{\partial \Omega}&:=\int_{\partial \Omega} f(S)\cdot  g(S)  \dd S \quad &&\textrm{ if }
f\cdot g\in L^1(\partial \Omega),\\
\langle g, f \rangle&:= \langle g, f \rangle_{X^*,X} \quad &&\textrm{ if } f\in X \textrm{ and }
g\in X^*.
\end{aligned}
\end{align*}
We also require the space  $\mathcal{C}_{\textrm{weak}}(0,T;
L^p(\Omega))$  consisting of all $u \in L^{\infty}(0,T;
L^p(\Omega))$ such that $(u(t),\varphi) \in
\mathcal{C}([0,T])$ for all $\varphi \in
\mathcal{C}(\overline{\Omega})$.

We introduce certain subspaces (and their dual spaces) of the space of vector-valued Sobolev functions from $W^{1,p}(\Omega)^{d}$, which have zero normal component
on the boundary. First, we define, in the usual way, for any $p\in [1,\infty)$,
$$
\Lnd{p} := \overline{ \left\{ \bv \in \mathcal{D}(\Omega)^d; \, \diver_x \bv =0\right\}}^{\|\cdot\|_p},
 $$
 where by $\mathcal{D}(\Omega)$ we mean the set of all infinitely smooth functions with a compact support in the set $\Omega$.
Then, by $\mathcal{V}$ and $\mathcal{V}_{\diver}$ we denote
$$
\mathcal{V}:=\{\bv \in W^{d+2,2}(\Omega)^d; \, \bv \cdot \bn =0 \textrm{ on } \partial \Omega\}, \quad \mathcal{V}_{\diver}:= \mathcal{V}\cap \Lnd{2}.
$$
Note that $\mathcal{V}\subset W^{1,\infty} (\Omega)^d$ and therefore we can finally introduce the following spaces
for any $p\in [1,\infty)$, and $p':=p/(p-1)$:
\begin{equation*}
\begin{split}
\Wn{p} &:= \overline{\mathcal{V}}^{\|\cdot \|_{1,p}}, \; \Wnm{p'} := \left ( \Wn{p} \right )^*,\\
\Wnd{p} &:= \overline{\mathcal{V}_{\diver}}^{\|\cdot \|_{1,p}},\;
\Wndm{p'} := \left ( \Wnd{p} \right )^*.
\end{split}
\end{equation*}
Moreover, for any $\alpha>1$ we introduce the space $L^p_\alpha(\R_0):=\{\varphi\in L^p(\R_0): \int_{\R_0}r^\alpha\varphi^p(r)\dd r<
\infty\}$; analogously,  we let
$L^p_{\alpha}(\Omega \times \R_0):=\{\varphi\in L^p(\Omega \times \R_0): \int_\Omega\int_{\R_0}r^\alpha\varphi^p(x,r) \dd r \dd x <
\infty\}$, $\mathbb{R}_0:=(r_0,\infty)$.


\begin{Theorem}\label{T:theorem} Let  the assumptions  (A1)--(A6) be satisfied.
Then, for any $\Omega \in \mathcal{C}^{1,1}$, $T\in (0, \infty)$, any
$\bv_0$,  $\bef$, and any nonnegative $\psi_0$, $\phi_0$ satisfying
 \begin{equation}
\bv_0\in \Lnd{2}, \quad \bff\in L^{p'}(0,T; W^{-1,p'}_{\bn}), \quad \psi_0\in L^1(\Omega;L^1_{\theta_1^*}(\mathbb{R}_0))\cup L^2(\Omega; L^2_{3}(\mathbb{R}_0)), \quad  \phi_0\in L^\infty(\Omega), \label{T:6}
\end{equation}
with some $\theta_1^* > \theta$ and $\theta_1^*\ge 1$, there exists a quadruple $(q,\bv,\psi,\phi)$ and $q^*>1$ such that
\begin{align}
&q\in L^{q^*}(Q), \label{T:0} \\
&\bv \in  {\mathcal C}_{\textrm{weak}}(0,T; \Lnd{2})\cap L^p(0,T;W^{1,p}_{{\bn},\diver}),
\label{T:1}\\
&\psi\in L^{\infty}(0,T; L^2_3(\Omega \times \R_0))\cap L^2(0,T;L^2_{loc}(\R_0; W^{1,2}(\Omega)))\cap L^1(0,T;L^1(\Omega;L^1_{\theta_1^*}(\R_0))),\label{T3}\\
&\phi\in L^\infty((0,T)\times \Omega)\cap L^2(0,T;W^{1,2}(\Omega)), \label{T3a}
\end{align}
with $\partial_t \bv \in L^{q^*}(0,T; W^{-1, q^*}_{\bn, \diver})$,
$\partial_t\psi \in L^{q^*}(0,T; (W^{1,2}(\mathbb{R}_0\times \Omega)\cap W^{1,(q^*)'}(\mathbb{R}_0\times \Omega))^*)$
and $\partial_t\phi\in L^2(0,T; W^{-1,2}(\Omega))$, which attains the initial conditions in the following sense:
\begin{align}
&\lim_{t \to 0_+}\|\bv(t)-\bv_0\|_2^2 + \|\phi(t)-\phi_0\|_2^2 + \|\psi(t)-\psi_0\|_{2}^2 =0, \label{T:4}
\end{align}
and solves, for almost all $~t\in (0,T)$, the equations \eqref{NS}--\eqref{BC-psi} in the following sense:
for all $\bw\in W^{1,1}_{\bn}$ such that $\nabla_x\bw\in L^{\infty}(\Omega)^{d\times d}$,
\begin{align}
\begin{split}
&\langle\partial_t\bv ,\bw\rangle +
(\obS,\nabla_x\bw)
 - (\bv\otimes \bv , \nabla_x\bw) +\alpha^*(\bv , \bw)_{\partial\Omega}
 = \langle\bff,
 \bw\rangle + (q,\diver_x \bw)\,;
\end{split} \label{T:3}
\end{align}
for all $\varphi\in W^{1,\infty}(\Omega; \mathcal{D}(\mathbb{R}))$,
\begin{align}
\begin{split}
&\langle\partial_t \psi,\varphi\rangle - (\bv \psi,\nabla_x \varphi) -(\partial_r(\tau \varphi), \phi\psi)+(A(r)\nabla_x\psi, \nabla_x\phi)
\\&\qquad= -(\beta \psi,\varphi) +2\left(\int_r^{\infty}\beta \kappa \psi \dd\tilde r,\varphi\right);
\end{split}\label{T:3a}
\end{align}
and for all $z\in  W^{1,2}(\Omega)$,
\begin{align}
\begin{split}\label{T:3b}
&\langle \partial_t \phi, z\rangle-(\bv\phi,\nabla_x z)
+A_0(\nabla_x\phi,\nabla_x z)
\\
&\quad =-\left(\phi\int_{r_0}^\infty\partial_r (r\tau)\psi \dd r, z\right) + 2\left( \int_0^{r_0} r \int_{r_0}^{\infty}\beta\kappa\psi \dd\tilde r \dd r,z\right).\\
&\qquad  \end{split}
\end{align}
\end{Theorem}


A key difficulty in the mathematical analysis of the problem under consideration is that the partial differential equation (\ref{eq:psi}), governing the evolution of the distribution of polymer chains, is a hyperbolic equation with respect to the variable $r$ with a nonlocal term, which is nonlinearly coupled to the evolution equations (\ref{NS}) and (\ref{eq:phi}).

\section{Analytical framework}
An essential part of the existence proof in the fluid part of the problem relies on Lipschitz approximations of Bochner functions taking values in Sobolev spaces.
We recall from~\cite{BuGwMaSw2012} the following lemma, which collects the properties of the approximation in a simplified setting
(omitting the generality of Orlicz spaces used in~\cite{BuGwMaSw2012}).

\begin{Lemma}\label{Lip-Lem}
Let $\Omega \subset \mathbb{R}^d$, $d \in \{2,3\}$, be a bounded open set, let
$T>0$ be the length of the time interval and suppose that $\frac{1}{p}+\frac{1}{p'}=1$, with $p \in (1,\infty)$.
For any function  $\bH$
and arbitrary sequences  $\{\bu^n\}_{n=1}^\infty$ and
$\{\bH^n\}_{n=1}^{\infty}$, we consider
$$a^n:=|\bH^n| +|\bH|\quad{\rm and}\quad b^n:= |\bD_x \bu^n|,$$
   such that, for a certain $C^*>1$,
\begin{equation}
\begin{split}
&\int_{Q} |a^n|^{p'}+|b^n|^p\dd x \dd t + \mbox{\rm ess.sup}_{t\in (0,T)}\|\bu^n(t)\|_2^2 \le C^*,\\
&\bu^n \to \b0 \quad \textrm{ a.e. in } Q:=(0,T)\times \Omega. \label{A1}
\end{split}
\end{equation}
In addition, let $\{\bG^n\}_{n=1}^{\infty}$ and $\{\bef^n\}_{n=1}^\infty$ be such that $\bG^n$ is symmetric and
\begin{alignat}{2}
\bG^n & \to \b0 \qquad&&\textrm{strongly in } L^1(Q)^{d\times d}, \label{bG}\\
\bef^n & \to \b0 \qquad&&\textrm{strongly in } L^1(Q)^d, \label{bf}
\end{alignat}
and suppose that the following identity holds in $\mathcal{D}'(Q)^{d}$:
\begin{equation}
\partial_t\bu^n + \diver_x (\bH^n-\bH + \bG^n) = \bef^n. \label{distr}
\end{equation}
Then, there exists a $\beta>0$ such that, for arbitrary $Q_h
\subset \subset Q$ and for arbitrary
 $\lambda^*\in(p^{\frac{1}{p-1}},\infty)$  and arbitrary $k\in
\mathbb{N}$,  there exists a sequence of
$\{\lambda^n_k\}_{n=1}^\infty$, a sequence of open sets
$\{E^n_k\}_{n=1}^\infty$, $E^n_k\subset Q$, and a sequence
$\{\bu^{n,k}\}_{n=1}^{\infty}$ bounded in
$L^\infty_{loc}(0,T;W^{1,\infty}_{loc}(\Omega)^d)$, such that, for any $1\le s<\infty$,
\begin{alignat}{2}
\lambda^n_k &\in \bigg[\lambda^*, p^{-\frac{p^k-1}{p-1}} (\lambda^*)^{p^k}\bigg], \qquad&&\forall n\in \mathbb{N},\label{lambda}\\
\bu^{n,k} &\to \b0 \qquad&&\textrm{strongly in } L^{ s}(Q_h)^d, \label{strong}\\
 \|\bD_x(\bu^{n,k})\|_{L^\infty (Q_h)}&\le C(h,\Omega) \lambda^n_k,\label{norm}\\
\bu^{n,k}&=\bu^n \qquad&&\textrm{in } Q_h\setminus E^n_k, \label{equalit}\\
\limsup_{n\to \infty}|Q_h\cap E^n_k|&\le C(h,\Omega) \frac{C^*}{(\lambda^*)^p}. \qquad &&~ \label{meas}
\end{alignat}
Moreover,
\begin{align}
&\limsup_{n\to \infty} \int_{Q_h\cap E^n_k}\!\!\!\!\!\!\!
\left(|\bH^n|+|\bH|\right) |\bD_x(\bu^{n,k}) |\dd x \dd t
\le C(h,C^*)\left (\frac{1}{(\lambda^*)^{p-1}} + \frac{1}{k^\beta}\right),\label{fest}
\end{align}
and the following bound holds for all $g\in \mathcal{D}(Q_h)$:
\begin{align}
&-\liminf_{n\to \infty} \int_0^T \langle \partial_t\bu^n, \bu^{n,k}
g \rangle \dd t \le C(g,h,C^*)\left
(\frac{1}{(\lambda^*)^{p-1}} + \frac{1}{k}\right)^\beta.
\label{tdercon}
\end{align}
\end{Lemma}

\section{Uniform a~priori estimates}\label{apriori}
In this section, we shall formally derive the uniform estimates that will play a crucial role in the proof of existence of weak solutions to the problem. All of the statements below are valid for sufficiently smooth solutions. First, we recall a minimum principle. However, we will prove it for the following slightly modified problem (we shall only indicate the dependence of functions on
the variable $r$, except in cases when it is necessary to emphasize the dependence on the other independent variables as well):
\begin{align}\label{eq:psi-m}
\partial_t \psi+ \bv \cdot \nabla_x\psi  -A\Delta_x\psi&= -\beta\psi  -\tau \phi\partial_r\psi_++2\int_r^{\infty}\beta(\tilde r,\cdot)\kappa(r,\tilde r)\psi_+(\cdot,\tilde r) \dd\tilde r,\\
\label{eq:phi-m}
\partial_t \phi+\bv \cdot \nabla_x\phi
-A_0\Delta_x\phi &=-\phi\int_{r_0}^\infty\partial_r (r\tau)\psi_+ \dd r + 2 \int_0^{r_0} r \int_{r_0}^{\infty}\beta(\tilde r,\cdot)\kappa(r,\tilde r)\psi_+(\cdot,\tilde r) \dd\tilde r \dd r,
\end{align}
where $\psi_+:=\max(\psi,0)$.
\begin{Lemma}[Minimum principle for $\psi$ and $\phi$]\label{L-MP}
Let $\psi_0$ and $\phi_0$ be nonnegative; then, the functions $\psi$ and $\phi$ that solve the coupled system \eqref{eq:psi-m} and \eqref{eq:phi-m} are also nonnegative.
\end{Lemma}
It directly follows from Lemma \ref{L-MP} that we can replace $\psi_+$ by $\psi$ and then \eqref{eq:psi-m}, \eqref{eq:phi-m} reduce to \eqref{eq:psi}, \eqref{eq:phi}.
\begin{proof}
We test the equation \eqref{eq:psi-m} by $\psi_{-}$, where $\psi_{-}=\min(\psi,0)$ and integrate with respect to $x$ and $r$. Using the Neumann boundary condition on $\psi$ we get, after integration by parts, that
\begin{equation}\label{L1:E1}
\begin{split}
&\frac12 \frac{\dd}{\dd t} \|\psi_{-}(t)\|^2_2+\int_{\Omega}\int_{r_0}^{\infty}\left(A(\cdot)|\nabla \psi_{-}|^2 +\beta(r,\cdot)|\psi_{-}|^2\right)\dd r\dd x \\
&=\frac12\int_{\Omega}\int_{r_0}^{\infty} (\diver_x \bv) |\psi_{-}|^2+\tau \phi \psi_{-}\partial_r\psi_+  \dd r \dd x +2\int_\Omega\int_{r_0}^{\infty}\psi_{-}\int_r^{\infty}\beta(\tilde r,\cdot)\kappa(r,\tilde r)\psi_+(\cdot,\tilde r) \dd\tilde r\dd r \dd x,
\end{split}
\end{equation}
where we have used the fact that $\bv \cdot \bn =0$ on $\partial \Omega$. Then, the first term on the right-hand side is identically zero since $\diver_x \bv =0$. The same is true for the second term on the right-hand side since $\psi_{-}\partial_r \psi_+ = 0$. Finally, since $\beta$ is nonnegative, we see that the last term is nonpositive. Consequently, we get
$$
 \frac{\dd}{\dd t} \|\psi_{-}(t)\|^2_2\le 0,
$$
and since we have assumed that $\psi_0\ge 0$, we deduce that $\psi\ge 0$ almost everywhere. Using a similar procedure, we obtain an identical result also for $\phi$.
\end{proof}
Our next estimate is a maximum principle for $\phi$.
\begin{Lemma}[Maximum principle for $\phi$]\label{L-Max}
There exists a constant $C$, depending only on $K$, such that, if $\phi_0 \in L^{\infty}(\Omega)$, then
\begin{equation}
\mbox{\rm ess.sup}_{t \in (0,T)}\|\phi(t)\|_{\infty} \le \max(K^2,\|\phi_0\|_{\infty}).\label{L2:E1}
\end{equation}
\end{Lemma}
\begin{proof}
We begin the proof with a pointwise bound on the right-hand side of \eqref{eq:phi}. Using the nonnegativity of $\psi$ and $\phi$ (Lemma~\ref{L-MP}) and the assumptions on $\tau$ and $\beta$ and the definition of $\kappa$, we observe that
$$
\begin{aligned}
&-\phi\int_{r_0}^\infty\partial_r (r\tau)\psi \dd r + 2 \int_0^{r_0} r \int_{r_0}^{\infty}\beta(\tilde r,\cdot)\kappa(r,\tilde r)\psi(\tilde r) \dd\tilde r \dd r \\
&\qquad \le -\phi\int_{r_0}^\infty (\tau + r\partial_r \tau)\psi \dd r + K\int_0^{r_0} r \int_{r_0}^{\infty}\frac{\psi(\tilde r)}{\tilde r} \dd\tilde r \dd r\\
&\qquad \le -K^{-1}r_0\phi\int_{r_0}^\infty \psi \dd r + K\int_0^{r_0} r \int_{r_0}^{\infty}\frac{\psi(\tilde r)}{r_0} \dd\tilde r \dd r\\
&\qquad = -K^{-1}r_0\phi\int_{r_0}^\infty \psi \dd r + \frac{Kr_0}{2}\int_{r_0}^{\infty}\psi \dd r\\
&\qquad \le  -K^{-1}r_0\left(\int_{r_0}^\infty \psi(r) \dd r\right) \left(\phi - K^2\right)\\
&\qquad \le  -K^{-1}r_0\left(\int_{r_0}^\infty \psi(r) \dd r\right) \left(\phi - M\right),
\end{aligned}
$$
where $M$ is defined as $M:=\max(K^2, \|\phi_0\|_{\infty})$. Hence, by multiplying \eqref{eq:phi-m} with $(\phi-M)_+$ and integrating over $\Omega$, we get (using integration by parts, the Neumann data and the fact that the velocity is a solenoidal function) that
$$
\frac{\dd}{\dd t}\|(\phi-M)_+\|_2^2\le 0.
$$
Consequently, we immediately arrive at \eqref{L2:E1}.
\end{proof}
The next result concerns conservation of mass.
\begin{Lemma}[Conservation of mass]\label{L-mass}
Let the pair of functions $(\psi,\phi)$ be a solution to \eqref{eq:psi}, \eqref{eq:phi}; then, the following identity holds:
\begin{equation}\label{con_mass}
\frac{\dd}{\dd t}\left[\int_\Omega\phi(t,x) \dd x+\int_{r_0}^\infty r\int_\Omega\psi(t,x,r) \dd x \dd r  \right]=0.
\end{equation}
Consequently, for a.a. $t\in (0,T)$,
\begin{equation}\label{A-est-1}
 \int_\Omega\phi(t,x) \dd x+\int_{r_0}^\infty r\int_\Omega\psi(t,x,r) \dd x \dd r  = \int_\Omega\phi_0(x) \dd x+\int_{r_0}^\infty r\int_\Omega\psi_0(x,r) \dd x \dd r=:E_0.
\end{equation}
\end{Lemma}
\begin{proof}
First, we integrate~\eqref{eq:psi} over $\Omega$. Since $\div_x\bv=0$, the transport term is equal to $\div_x(\bv \psi)$ and then both the transport and diffusion terms vanish thanks to $\bv \cdot \bn = 0$ on $\partial \Omega$ and the zero Neumann boundary condition for $\psi$. Hence, we obtain
\begin{equation}\label{psi-cal}
\begin{split}
\partial_t \int_\Omega\psi&(t,x,r)\dd x
+ \int_\Omega \phi(t,x) \tau \partial_r\psi(t,x,r) \dd x
\\=& -\int_\Omega\beta(r, \cdot)\psi(t,x,r) \dd x +2\int_\Omega\int_r^{\infty}\kappa(r,\tilde r)\beta(\tilde r,\cdot)\psi(t,x,\tilde r)\dd \tilde r \dd x.
\end{split}
\end{equation}
In the next step we multiply~\eqref{psi-cal}  by $r$ and integrate over the interval $(r_0,\infty)$. Consequently,
\begin{equation}\label{psi-cal2}
\begin{split}
&\frac{\dd}{\dd t}\int_{r_0}^{\infty} r \int_\Omega\psi(t,x,r)\dd x \dd r
+ \int_{r_0}^{\infty}\int_\Omega r\tau  \phi(t,x) \partial_r\psi(t,x,r) \dd x \dd r\\
&\qquad = -\int_{r_0}^{\infty} \int_\Omega r\beta(r, \cdot)\psi(t,x,r) \dd x \dd r +2\int_{r_0}^{\infty}\int_\Omega r\int_r^{\infty}\kappa(r,\tilde r)\beta(\tilde r,\cdot)\psi(t,x,\tilde r)\dd\tilde r  \dd x \dd r.
\end{split}
\end{equation}
Finally, assuming that $\psi$ vanishes sufficiently quickly at infinity\footnote{At the level of these formal computations, it is assumed that $\psi$ vanishes sufficiently rapidly at infinity; the argument will be made rigorous in Section \ref{proof} by fixing the function space in which the tripe $(\bv, \psi, \phi)$,  understood as a weak solution to the problem, is sought.}, we can integrate by parts with respect to $r$ (note that since $\tau(r_0)=0$ the second boundary term also vanishes); this yields the identity
\begin{equation}\label{psi-cal3}
\begin{split}
&\frac{\dd}{\dd t}\int_{r_0}^{\infty} r \int_\Omega\psi(t,x,r)\dd x \dd r
- \int_{r_0}^{\infty}\int_\Omega \partial_r(r \tau)  \phi(t,x) \psi(t,x,r) \dd x \dd r\\
&\qquad = -\int_{r_0}^{\infty} \int_\Omega r\beta(r, \cdot)\psi(t,x,r) \dd x \dd r +2\int_{r_0}^{\infty}\int_\Omega r\int_r^{\infty}\kappa(r,\tilde r)\beta(\tilde r,\cdot)\psi(t,x,\tilde r)\dd\tilde r  \dd x \dd r.
\end{split}
\end{equation}
Next we integrate~\eqref{eq:phi} over $\Omega$. Similarly as above, after integration by parts,  the transport and the diffusion terms vanish and we get
\begin{equation}\label{phi-cal}
\begin{split}
\frac{\dd}{\dd t} \int_\Omega\phi(t,x)\dd x &=
-\int_\Omega\int_{r_0}^\infty \phi(t,x)\partial_r(r \tau)\psi(t,x,r)\dd r \dd x \\
&\quad + 2\int_{\Omega} \int_0^{r_0} r \int_{r_0}^{\infty}\beta(\tilde r,\cdot)\kappa(r,\tilde r)\psi(t,x,\tilde r)
\dd{\tilde r} \dd r \dd x.
\end{split}
\end{equation}
Moreover, using the fact that $\kappa(r,\tilde r)=0$ for $\tilde r <r_0$, we can rewrite the last term to obtain the identity
\begin{equation}\label{phi-cal1}
\begin{split}
\frac{\dd}{\dd t} \int_\Omega\phi(t,x)\dd x &=
-\int_\Omega\int_{r_0}^\infty \phi(t,x)\partial_r(r \tau)\psi(t,x,r)\dd r \dd x \\
&\quad + 2\int_{\Omega} \int_0^{r_0} r \int_{r}^{\infty}\beta(\tilde r,\cdot)\kappa(r,\tilde r)\psi(t,x,\tilde r)\dd\tilde r \dd r\dd x.
\end{split}
\end{equation}
Thus, by summing \eqref{psi-cal3} and \eqref{phi-cal1} and using Fubini's theorem, we have that
\begin{equation}\label{psi-cal4}
\begin{split}
&\frac{\dd}{\dd t}\left(\int_\Omega\phi(t,x)\dd x + \int_\Omega\int_{r_0}^{\infty} r\psi(t,x,r) \dd r\dd x\right)\\
&\qquad = - \int_\Omega \int_{r_0}^{\infty} r\beta(r, \cdot)\psi(t,x,r) \dd r  \dd x +2\int_\Omega \int_{0}^{\infty} r\int_r^{\infty}\kappa(r,\tilde r)\beta(\tilde r,\cdot)\psi(t,x,\tilde r)\dd\tilde r   \dd r \dd x.
\end{split}
\end{equation}
Finally, we evaluate the last term. Changing the order of integration and using the fact that $\kappa(r,\tilde r)=0$ for $\tilde r \le r_0$, we deduce that
$$
\begin{aligned}
&2\int_\Omega \int_{0}^{\infty} r\int_r^{\infty}\kappa(r,\tilde r)\beta(\tilde r,\cdot)\psi(t,x,\tilde r)\dd\tilde r   \dd r \dd x\\
&\qquad=2\int_\Omega \int_{0}^{\infty}\int_{0}^{\infty} \chi_{\tilde{r}\ge r} r\kappa(r,\tilde r)\beta(\tilde r,\cdot)\psi(t,x,\tilde r)\dd\tilde r   \dd r \dd x\\
&\qquad=2\int_\Omega\int_{0}^{\infty}\int_{0}^{\tilde r}  r\kappa(r,\tilde r)\beta(\tilde r,\cdot)\psi(t,x,\tilde r)\dd r \dd\tilde r    \dd x \\
&\qquad =2\int_\Omega\int_{r_0}^{\infty}\left(\int_{0}^{\tilde r}  r\kappa(r,\tilde r)\dd r\right)\beta(\tilde r,\cdot)\psi(t,x,\tilde r)\dd\tilde r    \dd x
\overset{\eqref{kappa}}{=}\int_\Omega\int_{r_0}^{\infty}\tilde r\beta(\tilde r,\cdot)\psi(t,x,\tilde r)\dd\tilde r    \dd x.
\end{aligned}
$$
Consequently, we see that the right-hand side of \eqref{psi-cal4} is identically zero, and we deduce \eqref{con_mass}.
\end{proof}

We shall now develop further bounds on the function $\phi$.

\begin{Lemma}[Parabolic regularity of $\phi$]\label{L-par-psi}
Let $\psi_0$, $\phi_0$ be nonnegative, $\phi_0\in L^{\infty}(\Omega)$ and $E_0 < \infty$ (cf. \eqref{A-est-1}); then,
\begin{equation}
\int_{Q}|\nabla \phi|^2\dd x \dd t \le C(\|\phi_0\|_{\infty},E_0,K,T).\label{L4-E1}
\end{equation}
\end{Lemma}
\begin{proof}
Let us rewrite equation~\eqref{eq:phi} as follows:
\begin{equation}\label{eq-phi-F}
\partial_t\phi +\bv\cdot \nabla_x\phi-A_0\Delta\phi=F,
\end{equation}
where $F$ is the right-hand side of~\eqref{eq:phi}. First, we deduce a pointwise bound on $F$. Using the  minimum principle for $\psi$ in conjunction with the properties of $\beta$ and $\tau$, we get
\begin{equation}
\begin{split}\label{oneside}
F(t,x)&=- \phi(t,x)\int_{r_0}^\infty\partial_r (r\tau)\psi(t,x,r)\dd r + 2 \int_0^{r_0} r \int_{r_0}^{\infty}\beta(\tilde r,\cdot)\kappa(r,\tilde r)\psi(t,x,\tilde r)\dd\tilde r \dd r\\
&\le 2 \int_0^{r_0} r \int_{r_0}^{\infty}\beta(\tilde r,\cdot)\kappa(r,\tilde r)\psi(t,x,\tilde r)\dd\tilde r \dd r\le 2K \int_0^{r_0} r \int_{r_0}^{\infty}\kappa(r,\tilde r)\psi(t,x,\tilde r)\dd\tilde r \dd r\\
&=2K \int_{r_0}^{\infty} \left(\int_0^{r_0} r \kappa(r,\tilde r)\dd r\right) \psi(t,x,\tilde r)\dd\tilde r \leq K\int_{r_0}^{\infty} r_0\psi(t,x,\tilde r)\dd\tilde r\\
&\le K\int_{r_0}^{\infty} r\psi(t,x, r)\dd r.
\end{split}
\end{equation}
Thus, upon multiplying \eqref{eq-phi-F} by $\phi$, integrating over $\Omega$, using the nonnegativity of $\phi$, partial integration (noting that all boundary terms again vanish) together with \eqref{A-est-1} and Lemma \ref{L-Max} give:
$$
\begin{aligned}
\frac12 \frac{\dd}{\dd t} \|\phi(t)\|_2^2 + A_0\|\nabla \phi(t)\|_2^2 &\le K\int_{\Omega}\phi(t,x)\int_{r_0}^{\infty} r\psi(t,x, r)\dd r\dd x\\
&\le  K\|\phi(t)\|_{\infty}\int_{\Omega}\int_{r_0}^{\infty} r\psi(t,x, r)\dd r\dd x\\
&\!\!\!\!\!\!\!\overset{\eqref{L2:E1},\eqref{A-est-1}}{\le} KE_0\, \max(K^2,\|\phi_0\|_{\infty}).
\end{aligned}
$$
Hence \eqref{L4-E1} directly follows.  We note in passing that we have not explicitly indicated the dependence of the constant $C(\|\psi_0\|_{\infty}, E_0,K,T)$ appearing in \eqref{L4-E1} on $A_0$ and $|\Omega|$.
\end{proof}
The next step is to establish a bound on high-order moments of the function $\psi$.
\begin{Lemma}[High-order moments of $\psi$]\label{L-psi-moment}
Let $\psi$ be nonnegative and suppose that it satisfies \eqref{eq:psi}; then, for all $\alpha \ge 0$,
\begin{equation}\label{moment:E1}
\mbox{{\rm ess.sup}}_{t\in (0,T)}\int_{\Omega}\int_{r_0}^{\infty}r^{\alpha}\psi(t,x,r)\dd r\dd x \le C(\|\phi_0\|_{\infty}, \alpha,K,T)\int_{\Omega}\int_{r_0}^{\infty}r^{\alpha} \psi_0(x,r) \dd r  \dd x.
\end{equation}
\end{Lemma}
\begin{proof}
We first integrate \eqref{eq:psi} over $\Omega$ to get
$$
\begin{aligned}
\partial_t \int_{\Omega}\psi(t,x,r)\dd x&= -\int_{\Omega}\beta(r,\cdot)\psi(t,x,r)\dd x  +2\int_{\Omega}\int_r^{\infty}\beta(\tilde r,\cdot)\kappa(r,\tilde r)\psi(t,x,\tilde r)\dd\tilde r\dd x\\
&\qquad  -\int_{\Omega}\tau(r) \phi(t,x)\partial_r\psi(t,x,r) \dd x.
\end{aligned}
$$
We then multiply the result by $r^{\alpha}$, where $\alpha \ge 0$, and integrate over $(r_0,\infty)$ to deduce that
$$
\begin{aligned}
&\frac{\dd}{\dd t}\int_{\Omega}\int_{r_0}^{\infty}r^{\alpha}\psi(t,x,r)\dd r\dd x= -\int_{\Omega}\int_{r_0}^{\infty}r^{\alpha}\beta(r,\cdot)\psi(t,x,r)\dd r\dd x  \\
&\quad +2\int_{\Omega}\int_{r_0}^{\infty}r^{\alpha}\int_r^{\infty}\beta(\tilde r,\cdot)\kappa(r,\tilde r)\psi(t,x,\tilde r)\dd\tilde r\dd r\dd x-\int_{\Omega}\int_{r_0}^{\infty}\tau(r)r^{\alpha} \phi(t,x)\partial_r\psi(t,x,r) \dd r \dd x.
\end{aligned}
$$
Next, we evaluate the second term on the right-hand side. Using Fubini's theorem, we arrive at
$$
\begin{aligned}
2&\int_{r_0}^{\infty}r^{\alpha}\int_r^{\infty}\beta(\tilde r,\cdot)\kappa(r,\tilde r)\psi(t,x,\tilde r)\dd\tilde r\dd r=2\int_{r_0}^{\infty}\int_{r_0}^{\infty}r^{\alpha}\chi_{r\le \tilde r}\beta(\tilde r,\cdot)\kappa(r,\tilde r)\psi(t,x,\tilde r) \dd\tilde r  \dd r\\
&=2\int_{r_0}^{\infty}\int_{r_0}^{\tilde r}r^{\alpha}\beta(\tilde r,\cdot)\kappa(r,\tilde r)\psi(t,x,\tilde r)\dd r \dd\tilde r\\
&=2\int_{r_0}^{\infty}\left(\int_{0}^{\tilde r}r^{\alpha}\kappa(r,\tilde r)\dd r \right)\beta(\tilde r,\cdot)\psi(t,x,\tilde r)\dd\tilde r-2\int_{r_0}^{\infty}\left(\int_{0}^{r_0}r^{\alpha}\kappa(r,\tilde r)\dd r \right)\beta(\tilde r,\cdot)\psi(t,x,\tilde r) \dd\tilde r\\
&\overset{\eqref{kappa}}=\frac{2}{\alpha+1}\int_{r_0}^{\infty}r^{\alpha}\beta( r,\cdot)\psi(t,x,r)\dd r -\frac{2r_0^{\alpha+1}}{\alpha+1}\int_{r_0}^{\infty}\frac{\beta( r,\cdot)\psi(t,x, r)}{r}\dd r.
\end{aligned}
$$
Hence, we have that
\begin{equation*}
\begin{aligned}
&\frac{\dd}{\dd t}\int_{\Omega}\int_{r_0}^{\infty}r^{\alpha}\psi(t,x,r)\dd r\dd x+\frac{2r_0^{\alpha+1}}{\alpha+1}\int_{\Omega}\int_{r_0}^{\infty}\frac{\beta( r,\cdot)\psi(t,x, r)}{ r}\dd r\dd x\\
&\quad = \frac{1-\alpha}{\alpha+1}\int_{\Omega}\int_{r_0}^{\infty}r^{\alpha}\beta(r,\cdot)\psi(t,x,r)\dd r\dd x -\int_{\Omega}\int_{r_0}^{\infty}\tau(r)r^{\alpha} \phi(t,x)\partial_r\psi(t,x,r)\dd r \dd x.
\end{aligned}
\end{equation*}
Using the nonnegativity of $\psi$, the boundedness of $\phi$ and integration by parts in the last integral,
\begin{equation*}
\begin{aligned}
\frac{\dd}{\dd t}\int_{\Omega}\int_{r_0}^{\infty}r^{\alpha}\psi(t,x,r)\dd r\dd x&\le \frac{1-\alpha}{\alpha+1}\int_{\Omega}\int_{r_0}^{\infty}r^{\alpha}\beta(r,\cdot)\psi(t,x,r)\dd r\dd x \\
&\qquad +\int_{\Omega}\int_{r_0}^{\infty}\partial_r(\tau(r)r^{\alpha}) \phi(t,x)\psi(t,x,r)\dd r  \dd x\\
&\le C(\alpha,K)(1+\|\phi(t)\|_{\infty})\int_{\Omega}\int_{r_0}^{\infty}r^{\alpha} \psi(t,x,r)\dd r  \dd x,
\end{aligned}
\end{equation*}
where we have used the assumptions on $\tau$ and $\beta$. Thus, \eqref{moment:E1} follows by using Gronwall's lemma.
\end{proof}
Now we can prove the desired bounds on $\psi$.
\begin{Lemma}[Parabolic-hyperbolic estimates for $\psi$]\label{key-lemma}
Let $\psi$ solve \eqref{eq:psi}. Then, for all $p\in [3,\infty)$, the following inequality holds:
\begin{equation}\label{Key-Est}
\begin{split}
&\mbox{{\rm ess.sup}}_{t\in (0,T)}\int_\Omega \int_{r_0}^\infty r^p\psi^2(t,x,r)\dd r \dd x +\int_Q \int_{r_0}^\infty r^pA(r)|\nabla_x\psi(t,x,r)|^2\dd r\dd x\dd t \\
&\qquad \le C(p,K,\|\phi_0\|_{\infty})\int_\Omega \int_{r_0}^\infty r^p\psi_0^2(x,r)\dd r \dd x.
\end{split}
\end{equation}
\end{Lemma}
\begin{proof}
Let $\alpha(r)$ be an arbitrary nonnegative function. We multiply \eqref{eq:psi} by $\alpha(r)\psi(t,x,r)$ and integrate over $\Omega\times (r_0,\infty)$ to deduce that (note that the term containing $\bv$ again vanishes)
\begin{equation}\label{a}
\begin{split}
&\frac12 \frac{\dd}{\dd t}\int_\Omega \int_{r_0}^\infty\alpha(r)\psi^2(t,x,r)\dd r \dd x +\int_\Omega \int_{r_0}^\infty \alpha(r)A(r)|\nabla_x\psi(t,x,r)|^2\dd r\dd x
\\
&= -\int_\Omega \int_{r_0}^\infty \alpha(r)\beta(r, \cdot)\psi^2(t,x,r)  \dd r \dd x\\
&\quad+2\int_\Omega\int_{r_0}^\infty\int_r^{\infty}\beta(\tilde r,\cdot)\kappa(r,\tilde r)\psi(t,x,\tilde r)\dd\tilde r\,
\alpha(r)\psi(t,x,r)\dd r\dd x\\
&\quad- \frac12\int_\Omega \int_{r_0}^\infty \phi(t,x)\tau(r)\alpha(r)\partial_r\psi^2(t,x,r)\dd r \dd x.
\end{split}
\end{equation}
We begin by focusing on the evaluation of the first two terms on the right-hand side. By simple manipulations we deduce that
\begin{equation}\label{compute}
\begin{aligned}
Y(t,x)&:=-\int_{r_0}^\infty \left(\alpha(r)\beta(r, \cdot)\psi^2(t,x,r)  - 2\alpha(r)\psi(t,x,r)\int_r^{\infty}\beta(\tilde r,\cdot)\kappa(r,\tilde r)\psi(t,x,\tilde r)\dd\tilde r\right)
\dd r\\
&=-\int_{r_0}^\infty \alpha(r)\psi(t,x,r) \left(\beta(r, \cdot)\psi(t,x,r) -2\int_r^{\infty}\beta(\tilde r,\cdot)\kappa(r,\tilde r)\psi(t,x,\tilde r)\dd\tilde r\right)\dd r\\
&=-\int_{r_0}^\infty \alpha(r)\psi(t,x,r) \left(\beta(r,\cdot)\psi(t,x,r) -2\int_r^{\infty}\frac{\beta(\tilde r,\cdot)\psi(t,x,\tilde r)}{\tilde r}\dd\tilde r\right)\dd r\\
&=\int_{r_0}^\infty \frac{\alpha(r)\psi(t,x,r)}{r} \partial_r\left( r^{2}\int_r^{\infty}\frac{\beta(\tilde r,\cdot)\psi(t,x,\tilde r)}{\tilde r}\dd\tilde r\right)\dd r\\
&=:\int_{r_0}^\infty \frac{\alpha(r)\psi(t,x,r)}{r} \partial_rI(t,x,r)\dd r.
\end{aligned}
\end{equation}
Next, we rewrite $\psi$ in terms of $I(r)$ as follows:
$$
r\beta(r,\cdot)\psi(t,x,r)=-\partial_r I(t,x,r)+ 2r\int_r^{\infty}\frac{\beta(\tilde r,\cdot)\psi(t,x,\tilde r)}{\tilde r}\dd\tilde r=-\partial_r I(t,x,r)+ \frac{2I(t,x,r)}{r}.
$$
Hence, substituting this identity into \eqref{compute}, we deduce that
\begin{equation}\label{compute2}
\begin{aligned}
Y(t,x)&=\int_{r_0}^\infty \frac{\alpha(r)}{r^2\beta(r,\cdot)} r\beta(r,\cdot)\psi(t,x,r)\partial_rI(t,x,r)\dd r\\
&=\int_{r_0}^\infty \frac{\alpha(r)}{r^2\beta(r,\cdot)}\partial_rI(t,x,r)\left( -\partial_r I(t,x,r)+ \frac{2I(t,x,r)}{r}\right)\dd r\\
&=-\int_{r_0}^\infty \frac{\alpha(r)}{r^2\beta(r,\cdot)}|\partial_rI(t,x,r)|^2\dd r +\int_{r_0}^\infty \frac{\alpha(r)}{r^3\beta(r,\cdot)}\partial_rI^2(t,x,r)\dd r\\
&=-\int_{r_0}^\infty \frac{\alpha(r)}{r^2\beta(r,\cdot)}|\partial_rI(t,x,r)|^2\dd r -\int_{r_0}^\infty \partial_r\left(\frac{\alpha(r)}{r^3\beta(r,\cdot)}\right)I^2(t,x,r)\dd r\\
&\qquad - \frac{\alpha(r_0)}{r_0^3\beta(r_0,\cdot)}I^2(t,x,r_0).
\end{aligned}
\end{equation}
Thus, returning to \eqref{a}, we substitute $Y(t,x)$ into the first two terms on the right-hand side, and integrate by parts in the last term in \eqref{a} recalling that $\tau(r_0)=0$; this yields
\begin{equation}\label{b}
\begin{split}
&\frac{\dd}{\dd t}\int_\Omega \int_{r_0}^\infty\alpha(r)\psi^2(t,x,r)\dd r \dd x +2\int_\Omega \int_{r_0}^\infty \alpha(r)A(r)|\nabla_x\psi(t,x,r)|^2\dd r\dd x
\\
&\quad +2\int_{\Omega}\int_{r_0}^\infty \frac{\alpha(r)}{r^2\beta(r,\cdot)}|\partial_rI(t,x,r)|^2\dd r\dd x+2\int_{\Omega}\frac{\alpha(r_0)}{r_0^3\beta(r_0,\cdot)}I^2(t,x,r_0)\dd x\\
&= -2\int_{\Omega}\int_{r_0}^\infty \partial_r\left(\frac{\alpha(r)}{r^3\beta(r,\cdot)}\right)I^2(t,x,r)\dd r \dd x\\
&\quad +\int_\Omega \int_{r_0}^\infty \phi(t,x)\partial_r\left(\tau(r)\alpha(r)\right)\psi^2(t,x,r)\dd r \dd x.
\end{split}
\end{equation}
 By setting $\alpha(r):=r^3 \tilde{\beta}(r)$ with nonnegative $\tilde \beta$, using H\"{o}lder's inequality for the last term, we get
\begin{equation}\label{c}
\begin{split}
&\frac{\dd}{\dd t}\int_\Omega \int_{r_0}^\infty r^3 \tilde{\beta}(r)\psi^2(t,x,r)\dd r \dd x +2\int_\Omega \int_{r_0}^\infty r^3 \tilde{\beta}(r)A(r)|\nabla_x\psi(t,x,r)|^2\dd r\dd x
\\
&\quad +2\int_{\Omega}\int_{r_0}^\infty \frac{r \tilde{\beta}(r)}{\beta(r,\cdot)}|\partial_rI(t,x,r)|^2\dd r\dd x+2\int_{\Omega}\frac{ \tilde{\beta}(r_0)}{\beta(r_0,\cdot)}I^2(t,x,r_0)\dd x\\
&\le -2\int_{\Omega}\int_{r_0}^\infty \partial_r\left(\frac{ \tilde{\beta}(r)}{\beta(r,\cdot)}\right)I^2(t,x,r)\dd r \dd x\\
&\quad +\|\phi(t)\|_{\infty}\int_\Omega \int_{r_0}^\infty \left|\partial_r \tau(r)+\frac{\tau(r)\partial_r \tilde{\beta}(r)}{\tilde{\beta}(r)}+\frac{3\tau(r)}{r}\right|r^3\tilde{\beta}(r)\psi^2(t,x,r)\dd r \dd x.
\end{split}
\end{equation}
Consequently, if we choose $\tilde \beta$ such that for all $(r,\bu,\bD) \in (r_0,\infty)\times \mathbb{R}^d\times \mathbb{R}^{d\times d}$ there holds
\begin{align}
\partial_r\left(\frac{ \tilde{\beta}(r)}{\beta(r,\bu,\bD)}\right)\ge 0, \qquad \left|\frac{\tau(r)\partial_r \tilde{\beta}(r)}{\tilde{\beta}(r)}\right|\le C(\tilde{\beta}),
\label{findb}
\end{align}
then it follows from our assumptions on $\tau$ (cf. \eqref{tau-a}) that
\begin{equation*}
\begin{split}
&\frac{\dd}{\dd t}\int_\Omega \int_{r_0}^\infty r^3 \tilde{\beta}(r)\psi^2(t,x,r)\dd r \dd x +2\int_\Omega \int_{r_0}^\infty r^3 \tilde{\beta}(r)A(r)|\nabla_x\psi(t,x,r)|^2\dd r\dd x\\
&\le C(K,\|\phi_0\|_{\infty})\int_\Omega \int_{r_0}^\infty r^3\tilde{\beta}(r)\psi^2(t,x,r)\dd r \dd x
\end{split}
\end{equation*}
and by Gronwall's lemma we get
\begin{equation}\label{d}
\begin{split}
&\mbox{{\rm ess.sup}}_{t\in (0,T)}\int_\Omega \int_{r_0}^\infty r^3 \tilde{\beta}(r)\psi^2(t,x,r)\dd r \dd x +\int_Q \int_{r_0}^\infty r^3 \tilde{\beta}(r)A(r)|\nabla_x\psi(t,x,r)|^2\dd r\dd x \dd t \\
&\qquad \le C(K,\|\phi_0\|_{\infty})\int_\Omega \int_{r_0}^\infty r^3\tilde{\beta}(r)\psi_0^2(x,r)\dd r \dd x.
\end{split}
\end{equation}
In what follows, we focus on finding $\tilde{\beta}$ that satisfies \eqref{findb}.
For any nondecreasing $\gamma(r)\ge 1$ we define
$$
\tilde{\beta}(r):= \gamma(r)\, {\rm e}^{\int_{r_0}^r \eta(\tau)\dd \tau},
$$
where $\eta$ is introduced in \eqref{dfeta}
and we check \eqref{findb}. For the second inequality in \eqref{findb}, we observe that
$$
\begin{aligned}
\frac{\partial_r \tilde{\beta}(r)}{\tilde{\beta}(r)}=\frac{\partial_r\gamma(r)}{\gamma(r)} + \eta(r);
\end{aligned}
$$
consequently, since $\tau(r)\le Cr$ and $\eta$ satisfies \eqref{eta-as}, it suffices to choose $\gamma$ such that
\begin{equation}
\label{gamma-choice1}
\frac{r\partial_r \gamma(r)}{\gamma(r)}\le C \qquad \textrm{ for all } r\in [r_0,\infty]
\end{equation}
to ensure the validity of the second inequality in \eqref{findb}. To check also the first inequality in \eqref{findb}, we deduce with the aid of  \eqref{dfeta} and the fact that $\gamma$ is nondecreasing that
$$
\begin{aligned}
\partial_r\left(\frac{ \tilde{\beta}(r)}{\beta(r,\bu,\bD)}\right)&=\frac{ \partial_r\tilde{\beta}(r)\beta(r,\bu,\bD)-\tilde{\beta}(r)\partial_r\beta(r,\bu,\bD)}{\beta^2(r,\bu,\bD)}=\frac{\tilde{\beta}(r)}{\beta(r,\bu,\bD)}
\left(\frac{ \partial_r\tilde{\beta}(r)}{\tilde{\beta}(r)}-\frac{\partial_r\beta(r,\bu,\bD)}{\beta(r,\bu,\bD)}\right)\\
&\ge \frac{\tilde{\beta}(r)}{\beta(r,\bu,\bD)}
\left(\frac{ \partial_r\tilde{\beta}(r)}{\tilde{\beta}(r)}-\eta(r)\right)=\frac{\tilde{\beta}(r)}{\beta(r,\bu,\bD)}
\frac{ \partial_r\gamma(r)}{\gamma(r)}\ge 0.
\end{aligned}
$$
Consequently, the first inequality in \eqref{findb} is satisfied. In addition, setting $\gamma(r):=(1+r)^{\alpha}$ with arbitrary $\alpha \ge 0$, we see that $\gamma$ is increasing and also that \eqref{gamma-choice1} holds. Moreover, with such a choice, it is not difficult to verify using the definition of $\tilde{\beta}$ that
$$
\gamma(r)\le \tilde{\beta}(r)\le C\gamma(r).
$$
By substituting this relation into \eqref{d} we obtain \eqref{Key-Est}.
\end{proof}

\section{Proof of the main Theorem} \label{proof}
This section is devoted to the rigorous proof of our main result. To this end, we introduce several levels of approximation. First, since the argument is based on the energy method, we add a regularizing term to \eqref{NS} as follows:
\begin{equation}\label{NSeps}
\begin{split}
\partial_t\bv(t,x)+ \div_x (\bv(t,x)\otimes \bv(t,x)) +\nabla_x q(t,x)- {\rm div}_x \bS(\tilde\psi(t,x),\bD_x \bv(t,x))+\varepsilon|\bv|^{2p'-2}\bv&= \bff,\\
{\rm div}_x \bv(t,x) &= 0.
\end{split}
\end{equation}
Next, we also regularize the equation  \eqref{eq:psi} for $\psi$ by adding the term $-\delta\partial_{rr} \psi$ to its left-hand side
and to improve the integrability we add the terms $\delta|\psi|^{2p-2}\psi$ and $\delta|\phi|^{2p-2}\phi$ to  \eqref{eq:psi}  and  \eqref{eq:phi} respectively. Moreover, in order to prove the minimum principle for $\psi$ we replace $\psi$ by its positive part $\psi_+$ in several places. Also, in order to justify the a priori estimates for the approximating problem, we replace $\phi$ by $T_k (\phi)$, where the truncation function $T_k(s)$ is defined as follows:
$$
T_k(s):=\min (|s|,k)\,\textrm{sign }s,\qquad s \in \mathbb{R}.
$$
Finally,  we introduce $r_{\infty}>r_0$ and consider the following equation in $Q\times (r_0, r_{\infty})$, supplemented by homogeneous Neumann boundary conditions on $(0,T)\times \partial \Omega \times (r_0,r_\infty)$
and $(0,T)\times \Omega \times \{r_0,r_\infty\}$:
\begin{equation}\label{eq:psieta}
\begin{split}
\partial_t \psi&(t,x,r) + \bv(t,x) \cdot \nabla_x\psi(t,x,r) +\tau(r) T_k(\phi(t,x))\partial_r\psi_+(t,x,r)-A(r)\Delta_x\psi(t,x,r) - \delta\partial_{rr}\psi\\
&+\delta|\psi(t,x,r)|^{2p-2}\psi(t,x,r)= -\beta(r, \bv,\bD_x\bv)\psi(t,x,r) +2\int_r^{r_\infty}\beta(\tilde r,\bv, \bD_x\bv)\kappa(r,\tilde r)\psi_+(t,x,\tilde r)\dd\tilde r.
\end{split}
\end{equation}
Similarly, we replace the integral over the semi-infinite interval $(r_0,\infty)$ by one over $(r_0,r_\infty)$ and change the corresponding terms in \eqref{eq:phi} as follows:
\begin{equation}\label{eq:phieta}
\begin{split}
&\partial_t \phi(t,x)+\bv(t,x) \cdot \nabla_x\phi(t,x)
-A_0\Delta_x\phi(t,x)+\delta |\phi(t,x)|^{2p-2}\phi(t,x)\\
& =-T_k(\phi(t,x))\int_{r_0}^{r_\infty}\partial_r (r\tau(r))\psi_+(t,x,r)\dd r + 2 \int_0^{r_0} r \int_{r_0}^{r_\infty}\beta(\tilde r,\bv, \bD_x\bv)\kappa(r,\tilde r)\psi_+(t,x,\tilde r)\dd\tilde r \dd r,
\end{split}
\end{equation}
subject to a homogeneous Neumann boundary condition on $(0,T) \times \partial\Omega$.
Having introduced these three levels of approximation, we then first let $\delta \to 0$ and $r_{\infty} \to \infty$, and, finally, we let $\varepsilon \to 0$.

\subsection{Existence of a solution to the $\varepsilon$-, $\delta$-, $r_{\infty}$-approximating problem}
The existence of a solution to the approximating problem follows from the following result, which we state without proof, as the argument, based on monotone operator theory and the Aubin--Lions lemma, is completely straightforward.
\begin{Lemma}\label{Ex1}
Let $\Omega \subset \mathbb{R}^d$, $d \in \{2,3\}$, be a Lipschitz domain, $r_{\infty}\in (r_0,\infty)$ and $T>0$. Assume that (A1)--(A5) are satisfied. Moreover, let $\bv_0 \in L^2_{\bn,\diver}(\Omega)$, $\phi_0\in L^2(\Omega)$, $\psi_0\in L^2(\Omega \times (r_0,r_{\infty}))$ and $\bef \in L^{p'}(0,T; W^{-1,p'}_{\bn ,\diver})$. Then, there exists a triple $(\bv,\psi,\phi)$ such that
\begin{align}
\bv&\in \mathcal{C}(0,T; L^2(\Omega)^d) \cap L^p(0,T; W^{1,p}_{\bn,\diver}) \cap L^{2p'}(0,T; L^{2p'}(\Omega)^d),\label{SP1}\\
\partial_t \bv &\in \left(L^p(0,T; W^{1,p}_{\bn,\diver}) \cap L^{2p'}(0,T; L^{2p'}(\Omega)^d)\right)^*,\label{SP2}\\
\psi &\in \mathcal{C}(0,T; L^2(\Omega \times (r_0,r_{\infty})))\cap L^2(0,T; W^{1,2}(\Omega \times (r_0,r_{\infty})))\cap L^{2p}(Q\times (r_0,r_{\infty})),\label{SP3}\\
\partial_t\psi &\in \left( L^2(0,T; W^{1,2}(\Omega \times (r_0,r_{\infty})))\cap L^{2p}(Q\times (r_0,r_{\infty}))\right)^*,\label{SP4}\\
\phi &\in \mathcal{C}(0,T; L^2(\Omega ))\cap L^2(0,T; W^{1,2}(\Omega ))\cap L^{2p}(Q),\label{SP5}\\
\partial_t \phi &\in \left(L^2(0,T; W^{1,2}(\Omega ))\cap L^{2p}(Q)\right)^*,\label{SP6}
\end{align}
satisfying, for all $\bw \in L^p(0,T; W^{1,p}_{\bn,\diver}) \cap L^{p'}(Q)^d$,
\begin{equation}\label{NSeps-wf}
\begin{split}
&\int_0^T\langle\partial_t\bv,\bw \rangle -\int_Q \bv\otimes \bv \cdot \nabla_x \bw +\bS(\tilde\psi(t,x),\bD_x \bv(t,x))\cdot \nabla_x \bw+\varepsilon|\bv|^{2p'-2}\bv\cdot \bw \dd x\dd t\\
&= -\alpha^*\int_0^T \int_{\partial \Omega}\bv \cdot \bw\dd S +\int_0^T\langle \bff,\bw\rangle \dd t;
\end{split}
\end{equation}
for all $\omega \in L^2(W^{1,2}(\Omega\times (r_0,r_{\infty})))\cap L^{2p}(Q\times (r_0,r_{\infty}))$,
\begin{equation}\label{eq:psietawf}
\begin{split}
&\int_0^T \langle \partial_t \psi, \omega\rangle\dd t  + \int_{Q\times (r_0,r_{\infty})} \!\!\!\!\!\!\!\!\!\!\!\!\!\!\!\!-\psi \bv \cdot\nabla_x \omega  +T_k(\phi) \tau \omega \partial_r \psi_+ +A\nabla_x \psi \cdot \nabla_x \omega +\delta\partial_{r}\psi \partial_r \omega\dd r\dd x\dd t\\
&= \int_{Q\times (r_0,r_{\infty})} \!\!\!\!\!\!\!\!\!\!\!\!\!\!\!\!-\delta|\psi|^{2p-2}\psi\omega-\beta(\cdot, \bv, \bD_x \bv)\psi\omega +2\omega\int_r^{r_\infty}\beta(\tilde r,\bv, \bD_x\bv )\kappa(r,\tilde r)\psi_+(\tilde r,\cdot)\dd\tilde r\dd r\dd x\dd t;
\end{split}
\end{equation}
for all $\theta \in L^2(0,T; W^{1,2}(\Omega)) \cap L^{2p}(Q)$,
\begin{equation}\label{eq:phietawf}
\begin{split}
&\int_0^T\langle\partial_t \phi, \theta \rangle\dd t +\int_Q -\phi\bv \cdot \nabla_x\theta
+A_0\nabla_x \phi\cdot \nabla_x \theta+\delta |\phi|^{2p-2}\phi\, \theta\dd x\dd t\\
& =-\int_QT_k(\phi)\theta\int_{r_0}^{r_\infty}\partial_r (r\tau(r))\psi_+(r,\cdot)\dd r \dd x\dd t \\
&\quad + 2 \int_{Q}\int_0^{r_0} \theta r \int_{r_0}^{r_\infty}\beta(\tilde r,\bv, \bD_x\bv)\kappa(r,\tilde r)\psi_+(\tilde r,\cdot)\dd\tilde r \dd r\dd x\dd t;
\end{split}
\end{equation}
and fulfilling, in addition,
$$
\bv(0)=\bv_0, \qquad \phi(0)=\phi_0,\qquad \psi(0)=\psi_0
$$
and
\begin{equation} \label{T-app}
{\bS(\tilde{\psi}(t,x),\bD_x \bv(t,x))}:= \nu(\tilde\psi(t,x), |\bD_x\bv(t,x)|)\,\bD_x\bv(t,x)
\end{equation}
with
\begin{equation}\label{psi-aver-app}
\tilde{\psi}(t,x):=\int_{r_0}^{r_{\infty}} \gamma(r)\psi(t,x,r)\dd r.
\end{equation}
\end{Lemma}

\subsection{A~priori estimates} \label{aprirori}
In this subsection, we derive bounds on the solution whose existence has been stated in Lemma~\ref{Ex1}. Some of the bounds will be independent of the order of approximation; if, on the other hand, something depends on one of the regularization parameters, this will be clearly indicated. Moreover, since we shall be following, step-by-step, the formal bounds developed in Section~\ref{apriori}, some of the details will be omitted for brevity. Also, to abbreviate the notation, we introduce $\Omega_{r_{\infty}}:=\Omega \times (r_0,r_{\infty})$. We formulate the result in the following lemma.
\begin{Lemma}\label{APestgood}
Let $(\bv,\psi,\phi)$ be a solution to \eqref{SP1}--\eqref{psi-aver-app} constructed in Lemma~\ref{Ex1}.  Then, the following energy identity holds for all $t\in (0,T)$:
\begin{align}
\frac12 \|\bv(t)\|_2^2 + \int_0^t \int_{\Omega}\varepsilon |\bv|^{2p'} + \bS \cdot \bD_x \bv\dd x\dd \tau + \alpha^*\int_0^t \int_{\partial \Omega}|\bv|^2 \dd S\dd \tau  &=\int_0^t \langle \bef, \bv \rangle \dd \tau + \frac12\|\bv_0\|_2^2. \label{1aprgood}
\end{align}
Consequently, we have the following uniform a~priori bound:
\begin{align}
&\mbox{{\rm ess.sup}}_{t\in (0,T)} \|\bv(t)\|_2 + \int_0^T \!\!\!\|\bv\|_{1,p}^p + \|\bS\|_{p'}^{p'} + \|\bv\|_{L^2(\partial \Omega)}^2 +\varepsilon\|\bv\|_{2p'}^{2p'} \dd t\! \le C \left(\|\bv_0\|_2^2 + \int_0^T\!\!\! \|\bef\|^{p'}_{-1,p'}\dd t\right)\!.\label{5aprgood}
\end{align}
Moreover, if $\psi_0$ and $\phi_0$ are nonnegative almost everywhere, then
\begin{align}
&\psi\ge 0 \textrm{ in } Q\times (r_0,r_{\infty}), \qquad \phi\ge 0 \textrm{ in } Q. \label{mp2}
\end{align}
In addition, if $\phi_0 \in L^{\infty}(\Omega)$, then the following uniform estimate holds:
\begin{align}
&\|\phi\|_{L^{\infty}(Q)}\le \max(K^2,\|\phi_0\|_{\infty}),\label{3aprgood}
\end{align}
and, for all $q\ge 3$, we have the following $\delta$-dependent bounds:
{\small
\begin{align}
&\int_0^T\|\phi\|_{1,2}\dd t \le  C(\|\phi_0\|_{\infty},K)\,\delta^{\frac12} r_{\infty}^{-3} + C(\|\phi_0\|_{\infty},K)  \int_{\Omega_{r_{\infty}}}r\psi_0\dd r\dd x  ~+~ C(\|\phi_0\|_{\infty},K)\,\delta^{\frac12}\int_{\Omega_{r_\infty}}r^3 \psi_0^2 \dd r \dd x ,\label{4aprgood}\\
&\mbox{{\small \rm ess.sup}}_{t\in (0,T)}\int_{\Omega_{r_{\infty}}}r^q\psi^2(t)\dd r \dd x\nonumber\\
&\qquad +\int_0^T \int_{\Omega_{r_{\infty}}} \!\!\!\!\!\!\!\! A(r)r^q|\nabla_x \psi|^2 + \delta r^q|\psi|^{2p} +\delta r^q|\partial_{r}\psi|^2 \dd r \dd x \dd t \le C(q){\rm e}^{C\delta T}\int_{\Omega_{r_{\infty}}}r^q\psi^2_0 \dd r \dd x,
 \label{6aprgood}\\
&\mbox{{\rm ess.sup}}_{t\in (0,T)} \int_{\Omega_{r_{\infty}}}r^{q-2}\psi(t)\dd r\dd x \le  C(\|\phi_0\|_{\infty},q,K)\left(\delta^{\frac12}  r_{\infty}^{q-6}+  \int_{\Omega_{r_{\infty}}}r^{q-2}\psi_0+\delta^{\frac12} r^q\psi_0^2\dd r\dd x\right).
\label{7aprgood}
\end{align}
}
\end{Lemma}
\begin{proof}
First, the energy identity \eqref{1aprgood} directly follows by setting $\bw:=\bv \chi_{[0,t]}$ in \eqref{NSeps-wf}. Then, the uniform estimate \eqref{5aprgood} is a consequence of the assumption \eqref{S} on $\bS$. To obtain the minimum principle for $\phi$, we mimic the proof of Lemma~\ref{L-MP}. Thus, we set $\theta:=\phi_{-}$ in \eqref{eq:phietawf}, where $\phi_{-}:=\min (0,\phi)$. Hence, using the divergence-free constraint on the velocity $\bv$, the fact that $\psi_{+}$, $\beta$ and $\tau$ are nonnegative and that $\tau$ is nondecreasing, we deduce that
$$
\frac{\dd}{\dd t} \|\phi_{-}\|_2^2 \le 0.
$$
Therefore, as $\phi_0$ is assumed to be nonnegative, we find that $\phi_{-}\equiv 0$ and $\phi\ge 0$ almost everywhere in $Q$. The proof of the inequality $\psi\ge 0$ is similar and is therefore omitted. Consequently, we can replace $\psi_{+}$ by $\psi$ in \eqref{eq:psietawf} and \eqref{eq:phietawf}.
In order to prove \eqref{3aprgood}, we mimic the proof of Lemma~\ref{L-Max}. To this end, we begin by noting that the right-hand side of \eqref{eq:phietawf} can be, for nonnegative $\theta$, bounded as follows:
$$
\begin{aligned}
&-\int_QT_k(\phi)\theta\int_{r_0}^{r_\infty}\partial_r (r\tau(r))\psi_+(t,x,r)\dd r \dd x\dd t \\
&\quad + 2 \int_{Q}\int_0^{r_0} \theta r \int_{r_0}^{r_\infty}\beta(\tilde r,\bv, \bD_x\bv)\kappa(r,\tilde r)\psi_+(t,x,\tilde r)\dd\tilde r \dd r\dd x\dd t\\
&\qquad \le -K^{-1}r_0\int_Q \theta \left(T_k(\phi) - M\right)\left(\int_{r_0}^{r_{\infty}} \psi(r)\dd r\right) \dd x\dd t,
\end{aligned}
$$
with $M:=\max(\|\phi_0\|_{\infty},K^2)$. Therefore, by setting $\theta:=(\phi-M)_+$ in \eqref{eq:phietawf}, using the fact that $\diver_x \bv=0$, the nonnegativity of $\phi$ and the above estimate, we arrive at
$$
\frac{\dd}{\dd t}\|(\phi-M)_+\|_2^2 \le 0.
$$
Since $\phi_0\le M$ almost everywhere in $\Omega$ the estimate \eqref{3aprgood} directly follows. Thus, in what follows, we assume that $k\ge M$ and therefore we can replace $T_k(\phi)$ by $\phi$ in \eqref{eq:psietawf} and \eqref{eq:phietawf}.

We see that all of the above estimates are independent of the order of the approximation. In what follows, we shall establish several estimates that depend on some of the regularization parameters, but the estimates will become uniform by applying the relevant limiting procedure. Note here that while in the proofs of Lemmas~\ref{L-mass}--\ref{key-lemma} we have formally used integration by parts with respect to $r$, assuming that $\psi$ vanishes sufficiently quickly at infinity, at the level of approximation, in a rigorous argument, such a procedure is not allowed. Fortunately, using the minimum principle for $\psi$ and the fact that $\tau(r_0)=0$, we can still integrate by parts in all of the desired terms at the cost of finally changing the equality sign to an inequality sign.

We start with the bounds that are similar to those in Lemma~\ref{key-lemma}. Hence, setting $\omega:=\alpha(r) \psi \chi_{[0,s]}$ in \eqref{eq:psietawf} with a suitable positive $\alpha(r)$, $r \in [r_0,\infty)$, to be selected below, we get (by recalling that $\Omega_{r_{\infty}}:=\Omega \times (r_0,r_{\infty})$) the following equality:
\begin{equation*}
\begin{split}
&\int_{\Omega_{r_{\infty}}}\frac{\alpha(r)}{2}(\psi^2(s)-\psi^2_0)\dd r\dd x +\int_0^s \int_{\Omega_{r_{\infty}}} \alpha(r)A(r)|\nabla_x \psi|^2 + \delta \alpha(r)|\psi|^{2p} +\delta \alpha |\partial_{r}\psi|^2\dd r\dd x\dd t \\
&= \int_0^s \int_{\Omega_{r_{\infty}}}-\beta(\cdot, \bv,\bD_x \bv) \alpha \psi^2 +2\alpha \psi\left(\int_r^{r_\infty}\beta(\tilde r,\bv,\bD_x\bv)\kappa(r,\tilde r)\psi(\tilde r)\dd\tilde r\right) \dd r\dd x\dd t \\
&\quad-\frac12\int_0^s\int_{\Omega_{r_{\infty}}}\phi \tau \alpha \partial_r \psi^2 \dd r\dd x\dd t
 -\delta\int_0^s\int_{\Omega_{r_{\infty}}}\alpha' \psi \partial_{r}\psi \dd r\dd x\dd t.
\end{split}
\end{equation*}
First, we use Young's inequality in order to absorb part of the last term into the left-hand side, to deduce the following inequality:
\begin{equation*}
\begin{split}
&\int_{\Omega_{r_{\infty}}}\frac{\alpha(r)}{2}(\psi^2(s)-\psi^2_0)\dd r\dd x +\int_0^s \int_{\Omega_{r_{\infty}}} \alpha(r)A(r)|\nabla_x \psi|^2 + \delta \alpha(r)|\psi|^{2p} +\frac{\delta}{2} \alpha |\partial_{r}\psi|^2\dd r\dd x\dd t \\
&\le \int_0^s \int_{\Omega_{r_{\infty}}}-\beta(\cdot, \bv, \bD_x \bv) \alpha \psi^2 +2\alpha \psi\left(\int_r^{r_\infty}\beta(\tilde r,\bv,\bD_x\bv)\kappa(r,\tilde r)\psi(\tilde r,\cdot)\dd\tilde r\right) \dd r\dd x\dd t \\
&\quad-\frac12\int_0^s\int_{\Omega_{r_{\infty}}}\phi \tau \alpha \partial_r \psi^2 \dd r\dd x\dd t
 +\delta\int_0^s\int_{\Omega_{r_{\infty}}}\frac{|\alpha'|^2}{2\alpha} \psi^2 \dd r\dd x\dd t.
\end{split}
\end{equation*}
Next, we integrate by parts in the third term on the right-hand side to deduce that
\begin{equation}\label{prI1}
\begin{split}
&\int_{\Omega_{r_{\infty}}}\frac{\alpha(r)}{2}(\psi^2(s)-\psi^2_0)\dd r\dd x +\int_0^s \int_{\Omega_{r_{\infty}}} \alpha(r)A(r)|\nabla_x \psi|^2 + \delta \alpha(r)|\psi|^{2p} +\frac{\delta}{2} \alpha |\partial_{r}\psi|^2\dd r\dd x\dd t \\
&\le \int_0^s \int_{\Omega_{r_{\infty}}}\!\!\!\!\!\!\! -\beta(\cdot,\bv,\bD_x \bv) \alpha \psi^2 +2\alpha \psi\left(\int_r^{r_\infty}\beta(\tilde r, \bv,\bD_x \bv)\kappa(r,\tilde r)\psi(\tilde r,\cdot)\dd\tilde r\right) \dd r\dd x\dd t\\
&\quad-\frac12\int_0^s\int_{\Omega}\phi(t,x) \left(\tau(r_{\infty}) \alpha(r_{\infty}) \psi^2(t,x,r_{\infty})-\tau(r_{0}) \alpha(r_{0}) \psi^2(t,x,r_{0})\right) \dd x\dd t\\
&\quad  +\delta\int_0^s\int_{\Omega_{r_{\infty}}}\frac{|\alpha'|^2}{2\alpha} \psi^2 \dd r\dd x\dd t+\frac12\int_0^s\int_{\Omega_{r_{\infty}}}\phi \partial_r (\tau \alpha) \psi^2 \dd r\dd x\dd t\\
&\le \int_0^s \int_{\Omega_{r_{\infty}}}\!\!\!\!\!\!\! -\beta(\cdot,\bv,\bD_x \bv) \alpha \psi^2 +2\alpha \psi\left(\int_r^{r_\infty}\beta(\tilde r,\bv,\bD_x \bv)\kappa(r,\tilde r)\psi(\tilde r,\cdot)\dd\tilde r\right) \dd r\dd x\dd t\\
&\quad  +\delta\int_0^s\int_{\Omega_{r_{\infty}}}\frac{|\alpha'|^2}{2\alpha} \psi^2 \dd r\dd x\dd t+\frac12\int_0^s\int_{\Omega_{r_{\infty}}}\phi \partial_r (\tau \alpha) \psi^2 \dd r\dd x\dd t,
\end{split}
\end{equation}
where for the second inequality we have used the fact that $\tau(r_0)=0$ and that $\phi$, $\tau$ and $\alpha$ are nonnegative. We then follow, step-by-step, the proof of Lemma~\ref{key-lemma}.  In particular, by setting $\alpha(r):=r^3 (1+r)^\gamma {\rm e}^{\int_{r_0}^r \eta(s)\dd s}$, with $\eta$ given by \eqref{dfeta} and nonnegative $\gamma$, we see that the first term on the right-hand side is nonpositive and therefore
\begin{equation}\label{prI2}
\begin{split}
&\int_{\Omega_{r_{\infty}}}\frac{\alpha(r)}{2}(\psi^2(s)-\psi^2_0)\dd r\dd x +\int_0^s \int_{\Omega_{r_{\infty}}} \alpha(r)A(r)|\nabla_x \psi|^2 + \delta \alpha(r)|\psi|^{2p} +\frac{\delta}{2} \alpha |\partial_{r}\psi|^2\dd r\dd x\dd t \\
&\le \delta\int_0^s\int_{\Omega_{r_{\infty}}}\frac{|\alpha'|^2}{2\alpha} \psi^2 \dd r\dd x\dd t+\frac12\int_0^s\int_{\Omega_{r_{\infty}}}\phi \partial_r (\tau \alpha) \psi^2 \dd r\dd x\dd t.
\end{split}
\end{equation}
Hence, by the assumption \eqref{tau-a}, the definition of $\alpha$ and Gronwall's lemma, we see that for all $q\ge 3$ (noting that we set $\gamma:=q-3$ in the definition of $\alpha$) we have
\begin{equation}\label{proI2}
\begin{split}
&\int_{\Omega_{r_{\infty}}}r^{q}\psi^2(t)\dd r\dd x +\int_0^T \int_{\Omega_{r_{\infty}}} r^{q}A(r)|\nabla_x \psi|^2 + \delta r^q|\psi|^{2p} +\frac{\delta}{2} r^q |\partial_{r}\psi|^2\dd r\dd x\dd t\\
&\le C(q){\rm e}^{C\delta T}\int_{\Omega_{r_{\infty}}}r^q\psi^2_0\dd r\dd x,
\end{split}
\end{equation}
which is nothing else than  \eqref{6aprgood}. Next, in order to estimate the high-order moments of $\psi$, for  any $q\ge 1$ we set $\omega:=r^q \chi_{[0,s]}$ in \eqref{eq:psietawf}  to obtain
\begin{equation}\label{lest}
\begin{split}
& \int_{\Omega_{r_{\infty}}}(r^q\psi(s)-r^q\psi_0)\dd r \dd x  + \int_0^s\int_{\Omega_{r_\infty}}\phi \tau r^q \partial_r \psi +\delta q\partial_{r}\psi r^{q-1}\dd r\dd x\dd t\\
&= -\int_0^{s}\int_{\Omega_{r_{\infty}}}\delta|\psi|^{2p-2}\psi r^q+\beta(\cdot, \bv,\bD_x \bv)\psi r^q -2r^q\left(\int_r^{r_\infty}\beta(\tilde r, \bv, \bD_x \bv)\kappa(r,\tilde r)\psi(\tilde r,\cdot )\dd\tilde r\right)\dd r\dd x\dd t.
\end{split}
\end{equation}
The first term on the right-hand side is nonpositive and can be therefore discarded. Next, we integrate by parts in the second term on the left-hand side and use the fact that all functions involved are nonnegative (and thus we can neglect their values at $r_{\infty}$); and, finally, we use Young's inequality in the last term on the left-hand side to deduce that
\begin{equation}\label{lest2}
\begin{split}
& \int_{\Omega_{r_{\infty}}}(r^q\psi(s)-r^q\psi_0) \dd r \dd x  \le \int_0^s\int_{\Omega_{r_\infty}} C(q)\delta^{\frac12}(\delta r^{q+2}|\partial_{r}\psi|^2+ r^{q-4})+\phi \partial_r(\tau r^q)  \psi \dd r \dd x \dd t\\
 &\quad +\int_0^s\int_{\Omega_{r_\infty}}-\beta\psi r^q +2r^q\left(\int_r^{r_\infty}\beta(\tilde r,\bv,\bD_x \bv)\kappa(r,\tilde r)\psi(\tilde r,\cdot)\dd\tilde r\right) \dd r \dd x \dd t.
\end{split}
\end{equation}
For the first term on the right-hand side, we use the estimate \eqref{6aprgood} and for the remaining parts we can follow, verbatim, the proof of Lemma~\ref{L-psi-moment}  (by using the assumption \eqref{tau-a} on $\tau$). Finally, we replace $q$ by $q-2$ to complete the proof of \eqref{7aprgood}.

Finally, we focus on \eqref{4aprgood}. Setting $\theta:=\phi$ in \eqref{eq:phietawf}, using the bound \eqref{3aprgood}, and the assumptions (A3) and (A4), we deduce that
\begin{equation}\label{conphi}
\begin{split}
\int_Q|\nabla_x \phi|^2 +\delta |\phi|^{2p} \dd x \dd t \le  C(A_0,K,\|\phi_0\|_{\infty})\int_{Q}\int_{r_0}^{r_\infty}r\psi \dd r  \dd x \dd t
\end{split}
\end{equation}
and \eqref{4aprgood} then follows by using \eqref{7aprgood} with $q=3$.
\end{proof}

\subsection{The limits $\delta\to 0_+$ and $r_{\infty} \to \infty$}
This subsection is devoted to passage to the limits $\delta\to 0_+$ and $r_{\infty} \to \infty$; i.e., we eliminate the presence of the elliptic regularization of $\psi$ and we pass from the bounded interval $(r_0,r_{\infty})$ to $(r_0,\infty)$. We shall use the notations $\Omega_{\infty}:=\Omega \times (r_0,\infty)$ and $Q_{\infty}:=(0,T)\times \Omega_{\infty}$. The associated existence result is formulated in the following lemma.
\begin{Lemma}\label{Ex2}
Let $\Omega \subset \mathbb{R}^d$, $d \in \{2,3\}$, be a Lipschitz domain and $T>0$. Assume that (A1)--(A5) are satisfied. Moreover, let $\bv_0 \in L^2_{\bn,\diver}$, $\bef \in L^{p'}(0,T; W^{-1,p'}_{\bn ,\diver})$, and let $\phi_0\in L^{\infty}(\Omega)$ and $\psi_0\in \mathcal{D}(\Omega_{\infty})$ be nonnegative; then, there exists a triple $(\bv,\psi,\phi)$ such that
\begin{align}
\bv&\in \mathcal{C}(0,T; L^2(\Omega)^d) \cap L^p(0,T; W^{1,p}_{\bn,\diver}) \cap L^{2p'}(0,T; L^{2p'}(\Omega)^d),\label{SP1_2}\\
\partial_t \bv &\in \left(L^p(0,T; W^{1,p}_{\bn,\diver}) \cap L^{2p'}(0,T; L^{2p'}(\Omega)^d)\right)^*,\label{SP2_2}\\
\psi &\in \mathcal{C}_{w}(0,T; L^2(\Omega_{\infty}))\cap L^2((0,T)\times (r_0,r_{\infty}); W^{1,2}(\Omega)),\label{SP3_2}\\
\partial_t\psi &\in \left( L^{2p'}(0,T; W^{1,2p'}(\Omega_{\infty} )\,\cap\, W^{1,2}(\Omega_{\infty} ))\right)^*,\label{SP4_2}\\
\phi &\in \mathcal{C}(0,T; L^2(\Omega ))\cap L^2(0,T; W^{1,2}(\Omega ))\cap L^{\infty}(Q),\label{SP5_2}\\
\partial_t \phi &\in \left(L^2(0,T; W^{1,2}(\Omega ))\right)^*,\label{SP6_2}
\end{align}
which satisfies, for all $\bw \in L^p(0,T; W^{1,p}_{\bn,\diver}) \cap L^{p'}(Q)^d$:
\begin{equation}\label{NSeps-wf_2}
\begin{split}
&\int_0^T\langle\partial_t\bv,\bw \rangle \dd t-\int_Q \bv\otimes \bv \cdot \nabla_x \bw +\bS(\tilde\psi(t,x),\bD_x \bv(t,x))\cdot \nabla_x \bw+\varepsilon |\bv|^{2p'-2}\bv\cdot \bw \dd x\dd t\\
&= -\alpha^*\int_0^T \int_{\partial \Omega}\bv \cdot \bw\dd S +\int_0^T\langle \bff,\bw\rangle \dd t;
\end{split}
\end{equation}
furthermore, for all $\omega \in L^{\infty}(0,T; W^{1,2}(\Omega_{\infty})\cap W^{1,\infty}(\Omega_{\infty}))\cap L^{2p}(Q_{\infty})$ fulfilling $w(\infty,\cdot)=0$ one has:
\begin{equation}\label{eq:psietawf_2}
\begin{split}
&\int_0^T \langle \partial_t \psi, \omega\rangle\dd t  + \int_{Q_{\infty}}-\psi \bv \cdot\nabla_x \omega  -\phi \partial_r(\tau \omega) \psi +A\nabla_x \psi \cdot \nabla_x \omega \dd r\dd x\dd t\\
&= \int_{Q_{\infty}} -\beta(\cdot, \bv, \bD_x \bv)\psi\omega +2\omega\int_r^{\infty}\beta(\tilde r,\bv, \bD_x\bv )\kappa(r,\tilde r)\psi(\tilde r,\cdot)\dd\tilde r\dd r\dd x\dd t;
\end{split}
\end{equation}
and, for all $\theta \in L^2(0,T; W^{1,2}(\Omega))$ one has:
\begin{equation}\label{eq:phietawf_2}
\begin{split}
&\int_0^T\langle\partial_t \phi, \theta \rangle\dd t +\int_Q -\phi\bv \cdot \nabla_x\theta
+A_0\nabla_x \phi\cdot \nabla_x \theta\dd x\dd t\\
& =-\int_Q \phi\theta\int_{r_0}^{\infty}\partial_r (r\tau(r))\psi(r,\cdot)\dd r \dd x\dd t \\
&\quad + 2 \int_{Q}\int_0^{r_0} \theta r \int_{r_0}^{\infty}\beta(\tilde r,\bv, \bD_x\bv)\kappa(r,\tilde r)\psi(\tilde r,\cdot)\dd\tilde r \dd r\dd x\dd t;
\end{split}
\end{equation}
with $\bv$, $\phi$ and $\psi$ attaining the initial data $\bv_0$, $\phi_0$ and $\psi_0$, respectively, in the sense that
\begin{equation}
\limsup_{t\to 0_+} \|\bv(t)-\bv_0\|_2^2 + \|\phi(t)-\phi_0\|_2^2 + \|\psi(t)-\psi_0\|_2^2=0,\label{nabyvat}
\end{equation}
and fulfilling \eqref{T-app}. Moreover, the solution satisfies \eqref{1aprgood}, \eqref{5aprgood}, \eqref{mp2} and \eqref{3aprgood}. In addition, for all $q\ge 3$ we have the following uniform estimates:
\begin{align}
&\int_0^T\|\phi\|_{1,2}\dd t \le  C(\|\phi_0\|_{\infty},K)  \int_{\Omega_{{\infty}}}r\psi_0\dd r\dd x
\dd t,\label{4aprgood_2}\\
&\mbox{{\rm ess.sup}}_{t\in (0,T)}\int_{\Omega_{{\infty}}}r^q\psi^2(t)\dd r\dd x+\int_0^T \int_{\Omega_{{\infty}}} A(r)r^q|\nabla_x \psi|^2\dd r\dd x\dd t \le C(q)\int_{\Omega_{{\infty}}}r^q\psi^2_0\dd r\dd x,
 \label{6aprgood_2}\\
&\mbox{{\rm ess.sup}}_{t\in (0,T)} \int_{\Omega_{{\infty}}}r^{q-2}\psi(t)\dd r\dd x \le C(\|\phi_0\|_{\infty},q,K)  \int_{\Omega_{{\infty}}}r^{q-2}\psi_0\dd r\dd x.
\label{7aprgood_2}
\end{align}
\end{Lemma}
\begin{proof}
To prove Lemma~\ref{Ex2} we use the existence result obtained in Lemma~\ref{Ex1} in conjunction with the uniform estimates derived in Lemma~\ref{APestgood}. Thus we set $\delta:=n^{-1}$ and $r_{\infty}:=\ln n$ in Lemma~\ref{Ex1} and denote the corresponding solution by ($\bv^n, \phi^n, \psi^n$). Using the estimates
\eqref{5aprgood}--\eqref{7aprgood}, the fact that $\psi_0$ is a smooth compactly supported function, and defining all functions involved to be identically  zero outside $(r_0,r_{\infty})$, we deduce the existence of a subsequence that we do not relabel and the existence of $(\bv,\phi,\psi, \bS)$, where $\phi$ and $\psi$ are nonnegative functions, such that
\begin{alignat}{2}
\bv^n &\rightharpoonup \bv &&\qquad\textrm{weakly in } L^p(0,T; W^{1,p}_{\bn, \diver})\cap L^{2p}(Q)^d,\label{con1}\\
\bS^n &\rightharpoonup \overline{\bS} &&\qquad\textrm{weakly in } L^{p'}(Q)^{d\times d},\label{con2}\\
\phi^n &\rightharpoonup^* \phi &&\qquad\textrm{weakly$^*$ in } L^{\infty}(Q),\label{con3}\\
\phi^n &\rightharpoonup \phi &&\qquad\textrm{weakly in } L^2(0,T; W^{1,2}(\Omega)),\label{con4}\\
\psi^n &\rightharpoonup \psi &&\qquad\textrm{weakly in } L^2(Q_{\infty})\cap L^2(0,T; L^2_{loc}(r_0,\infty; W^{1,2}(\Omega))).\label{con5}
\end{alignat}
Moreover, using the weak lower semicontinuity of the norm function, and the special choices of $\delta$ and $r_{\infty}$ made above, we deduce that the following uniform estimates hold for $q \geq 3$:
\begin{align}
&\|\phi\|_{L^{\infty}}(Q)\le C(K,\|\phi_0\|_{\infty}),\label{23aprgood_3}\\
&\mbox{{\rm ess.sup}}_{t\in (0,T)} \|\bv(t)\|_2 + \int_0^T \|\bv\|_{1,p}^p + \|\bS\|_{p'}^{p'} + \|\bv\|_{L^2(\partial \Omega)}^2 +\varepsilon\|\bv\|_{p'}^{p'} \dd t \le C(\bv_0,\bef),\label{5aprgood_3}\\
&\int_0^T\|\phi\|_{1,2}\dd t \le   C(\|\phi_0\|_{\infty},K)  \int_{\Omega_{{\infty}}}r\psi_0\dd r\dd x
\dd t,\label{4aprgood_3}\\
&\mbox{{\rm ess.sup}}_{t\in (0,T)}\int_{\Omega_{{\infty}}}r^q\psi^2(t)\dd r\dd x+\int_0^T \int_{\Omega_{{\infty}}} A(r)r^q|\nabla_x \psi|^2\dd r\dd x\dd t \le C(q)\int_{\Omega_{{\infty}}}r^q\psi^2_0 \dd r\dd x,
 \label{6aprgood_3}\\
&\mbox{{\rm ess.sup}}_{t\in (0,T)} \int_{\Omega_{\infty}}r^{q-2}\psi(t)\dd r\dd x \le  C(\|\phi_0\|_{\infty},q,K) \int_{\Omega_{\infty}}r^{q-2}\psi_0\dd r\dd x.
\label{7aprgood_3}
\end{align}

Our objective is to let $n\to \infty$ (and consequently $\delta \to 0_+$ and $r_{\infty} \to \infty)$) in \eqref{NSeps-wf}--\eqref{eq:phietawf}. To do so, we first observe that \eqref{NSeps-wf}--\eqref{eq:phietawf} imply the following weak convergence
results: 
\begin{alignat}{2}
\partial_t \bv^n &\rightharpoonup \partial_t \bv &&\qquad\textrm{weakly in } \left(L^p(0,T; W^{1,p}_{\bn, \diver})\cap L^{2p}(Q)^d\right)^*,\label{con6}\\
\partial_t \phi^n &\rightharpoonup \partial_t\phi &&\qquad\textrm{weakly in } (L^2(0,T; W^{1,2}(\Omega)))^*,\label{con7}\\
\partial_t \psi^n &\rightharpoonup \partial_t\psi &&\qquad\textrm{weakly in } (L^q(0,T; W^{1,2}(K)\cap W^{1,q}(K)))^*,\label{con8}
\end{alignat}
for sufficiently large $q$  and any $K:=\Omega \times (r_0, k)$ with arbitrary fixed $k\in \mathbb{R}$. In addition, red by noting the uniform (with respect to $r_{\infty}$) bounds \eqref{23aprgood_3}--\eqref{7aprgood_3}, we see that the weak limit $\partial_t \psi$ satisfies \eqref{SP4_2}.
Consequently, we can use the Aubin--Lions lemma to deduce strong convergence, and then extract subsequences still labelled by the index $n$ (i.e., without indicating the subsequences in our notation), such that
\begin{alignat}{2}
\bv^n &\to \bv &&\qquad\textrm{a.e. in  } Q,\label{con9}\\
\phi^n &\to  \phi &&\qquad\textrm{a.e. in  } Q.\label{con10}
\end{alignat}
However, we cannot claim the same convergence result for $\psi^n$ because of the lack of the compactness with respect to $r$. Nevertheless, we will show that, for arbitrary $z\in \mathcal{D}(r_0,\infty)$, we have
\begin{align}
\int_{r_0}^{r_\infty^n}z(r)\psi^n(t,x,r)\dd r &\to \int_{r_0}^{\infty}z(r)\psi(t,x,r)\dd r \qquad\textrm{strongly in  } L^2(Q).\label{con11}
\end{align}
First of all, it follows from \eqref{con5} that
\begin{align}
\int_{r_0}^{r_\infty^n}z(r)\psi^n(t,x,r)\dd r &\rightharpoonup \int_{r_0}^{\infty}z(r)\psi(t,x,r)\dd r \qquad\textrm{weakly in  } L^2(Q).\label{con12}
\end{align}
Hence the limit is defined uniquely. Next, denoting
\begin{equation}\label{deftilde}
\tilde{\psi}_{z}^n(t,x):=\int_{r_0}^{r_{\infty}} z(r)\psi^n(t,x,r)\dd r,
\end{equation}
we can set in \eqref{eq:psietawf} $\omega:=z(r)\varphi(t,x)$, where $\varphi \in \mathcal{D}(\Omega)$ is arbitrary, and for sufficiently large $n$ such that $\ln n$ is not in the support of $z$ we obtain the identity
\begin{equation}\label{eq:psireg}
\begin{split}
&\int_0^T \left \langle \partial_t \tilde{\psi}^n_{z}, \varphi \right \rangle\dd t  + \int_{Q} -\tilde{\psi}^n_{z} \bv \cdot\nabla_x \varphi  +\nabla_x \tilde{\psi}^n_{Az} \cdot \nabla_x \varphi\dd x\dd t\\
&= \int_{Q}\varphi \left(\int_{r_0}^{{\infty}}-\delta |\psi^n|^{2p-2}\psi^n z -\beta(\cdot,\bv^n,\bD_x \bv^n)\psi^n
+2z \left(\int_r^{\infty}\beta(\tilde r,\bv^n,\bD_x \bv^n)\kappa(r,\tilde r)\psi^n(\tilde r,\cdot )\dd\tilde r\right)\dd r\right)\dd x\dd t \\
& - \int_{Q} \phi \varphi \left(\int_{r_0}^{{\infty}} -\partial_r(\tau z) \psi^n\dd r\right) +\delta \varphi \left(\int_{r_0}^{{\infty}}\partial_{r}\psi^n \partial_r z\dd r\right)\dd x\dd t.
\end{split}
\end{equation}

\noindent
Consequently, using the a~priori estimates \eqref{5aprgood}--\eqref{7aprgood}, we see that, for sufficiently large $q$,
\begin{align}
\partial_t \tilde{\psi}^n_z&\rightharpoonup \partial_t \tilde{\psi}_z \qquad\textrm{weakly in  } L^q(0,T; W^{1,q}_0(\Omega)).\label{con14}
\end{align}
On the other hand, using \eqref{6aprgood} we also have that
$$
\int_0^T \|\tilde{\psi}^n_z\|_{1,2}^2\dd t \le C(z)\int_{Q_{\infty}} |\psi^n|^2 + A(r)r^3|\nabla_x \psi^n|^2 \dd r\dd x\dd t \le C,
$$
where the first inequality follows from the fact that $z$ has a compact support.
Hence the Aubin--Lions lemma completes the proof of \eqref{con11}.

Next, we shall apply a similar convergence argument to $\tilde{\psi}^n$, which is defined in \eqref{psi-aver}. Since $\gamma$ is a continuous function, we can find a sequence of $\gamma_{\varepsilon} \in \mathcal{D}(r_0,\infty)$ such that $\gamma_{\varepsilon}\nearrow \gamma$ almost everywhere and also locally in $\mathcal{C}(r_0,\infty)$. For such an approximation, we can however use the convergence result \eqref{con11} and obtain
\begin{align}
\int_{r_0}^{r_\infty^n}\gamma_{\varepsilon}(r)\psi^n(t,x,r)\dd r &\to \int_{r_0}^{\infty}\gamma_{\varepsilon}(r)\psi(t,x,r)\dd r \qquad\textrm{strongly in  } L^2(Q).\label{con18}
\end{align}
Hence, for the original function $\tilde{\psi}^n$ we have that
\begin{equation}\label{original}
\begin{split}
\int_Q |\tilde{\psi}^n -\tilde{\psi}|^2\dd x\dd t&:= \int_Q \left|\int_{r_0}^{\infty}\gamma(r)(\psi^n(r,t,x)-\psi(r,t,x))\dd r \right|^2\dd x\dd t\\
&\le 2\int_Q \left|\int_{r_0}^{\infty}(\gamma(r)-\gamma_{\varepsilon}(r))(\psi^n(r,t,x)-\psi(r,t,x))\dd r \right|^2\dd x\dd t\\
&\quad +2\int_Q \left|\int_{r_0}^{\infty}\gamma_{\varepsilon}(r)(\psi^n(r,t,x)-\psi(r,t,x))\dd r \right|^2\dd x\dd t.
\end{split}
\end{equation}
Consequently, using \eqref{con18}, we deduce that
\begin{equation}\label{original2}
\begin{split}
\lim_{n\to \infty}\int_Q |\tilde{\psi}^n -\tilde{\psi}|^2\dd x\dd t&\le 2\lim_{n\to \infty}\int_Q \left|\int_{r_0}^{\infty}(\gamma(r)-\gamma_{\varepsilon}(r))(\psi^n(r,t,x)-\psi(r,t,x))\dd r \right|^2\dd x\dd t\\
&\le 4\lim_{n\to \infty}\int_Q \left|\int_{r_0}^{r^*}(\gamma(r)-\gamma_{\varepsilon}(r))(\psi^n(r,t,x)-\psi(r,t,x))\dd r \right|^2\dd x\dd t\\
&\quad  +4\lim_{n\to \infty}\int_Q \left|\int_{r^*}^{\infty}|\gamma(r)||\psi^n(r,t,x)-\psi(r,t,x)|\dd r \right|^2\dd x\dd t.
\end{split}
\end{equation}
Thanks to the uniform convergence of $\gamma_{\varepsilon}$ on compact sets, we can also easily let $\varepsilon \to 0_+$ (with the aid of the uniform bound \eqref{7aprgood}) to get
\begin{equation}\label{original3}
\begin{split}
&\lim_{n\to \infty}\int_Q |\tilde{\psi}^n -\tilde{\psi}|^2\dd x\dd t\le 4\lim_{n\to \infty}\int_Q \left|\int_{r^*}^{\infty}|\gamma(r)||\psi^n(r,t,x)-\psi(r,t,x)|\dd r \right|^2\dd x\dd t\\
&\le 4\int_{r^*}^{\infty}\frac{dr}{r^2}\dd r \lim_{n\to \infty}\int_Q \int_{r^*}^{\infty}|\gamma(r)|^2 r^2|\psi^n(r,t,x)-\psi(r,t,x)|^2\dd r \dd x\dd t\\
&\le K(r^*)^{-1}\lim_{n\to \infty}\int_Q \int_{r^*}^{\infty} r^{2\theta + 2}(|\psi^n(r,t,x)|^2+|\psi(r,t,x)|^2)\dd r \dd x\dd t\le C(K,\psi_0)(r^*)^{-1},
\end{split}
\end{equation}
where we have used \eqref{growthgamma}, the fact that $\psi_0$ is compactly supported and the estimate \eqref{6aprgood}. Consequently, letting $r^*\to \infty$ we deduce that
\begin{align}
\tilde{\psi}^n &\to \tilde{\psi}\qquad\textrm{strongly in  } L^2(Q).\label{con20}
\end{align}

We now have all ingredients in place to complete the proof. First, it is standard to let $n\to \infty$ in \eqref{NSeps-wf} to deduce, for all $\bw \in L^p(0,T;W^{1,p}_{\bn,\diver})\cap L^{2p}(Q)^d$, that
\begin{equation}\label{NSeps-wf_3}
\begin{split}
&\int_0^T\langle\partial_t\bv,\bw \rangle -\int_Q \bv\otimes \bv \cdot \nabla_x \bw +\overline{\bS}\cdot \nabla_x \bw+\varepsilon|\bv|^{2p'-2}\bv\cdot \bw \dd x\dd t\\
&= -\alpha^*\int_0^T \int_{\partial \Omega}\bv \cdot \bw\dd S +\int_0^T\langle \bff,\bw\rangle \dd t,
\end{split}
\end{equation}
as well as the first limit in \eqref{nabyvat}. Consequently, setting $\bw:=\bv$, we obtain the following identity:
\begin{align}
\frac12 \|\bv(T)\|_2^2 + \int_0^T \int_{\Omega}\varepsilon |\bv|^{2p'} + \overline{\bS} \cdot \bD_x \bv\dd x\dd \tau + \alpha^*\int_0^T \int_{\partial \Omega}|\bv|^2 \dd S\dd \tau  &=\int_0^T \langle \bef, \bv \rangle \dd \tau + \frac12\|\bv_0\|_2^2. \label{1aprgood_2}
\end{align}
Then, using the weak lower semicontinuity of the norm function, the energy identity \eqref{1aprgood} with $t=T$ and the weak convergence result \eqref{con1}, we deduce that
\begin{align}
\limsup_{n\to \infty} \int_Q\bS(\tilde{\psi}^n,\bD_x \bv^n) \cdot \bD_x \bv^n\dd x\dd t \le \int_Q \overline{\bS}\cdot \bD_x \bv\dd x\dd t. \label{1aprgood_4}
\end{align}
Moreover, using \eqref{con20}, Lebesgue's dominated convergence theorem and the assumption \eqref{S}, we also see that
\begin{align}
\bS(\tilde{\psi}^n,\bD_x \bv)&\to \bS(\tilde{\psi},\bD_x \bv)\qquad\textrm{strongly in  } L^{p'}(Q)^{d\times d},\label{con24}
\end{align}
which, when combined with \eqref{1aprgood_4} and \eqref{con1}, \eqref{con2}, leads to
\begin{align}
\limsup_{n\to \infty} \int_Q(\bS(\tilde{\psi}^n,\bD_x \bv^n)-\bS(\tilde{\psi}^n,\bD_x \bv)) \cdot (\bD_x \bv^n-\bD_x \bv)\dd x\dd t =0.
\end{align}
The strict monotonicity assumption \eqref{S} then implies that there is a subsequence such that
\begin{align}
\bD_x \bv^n\to \bD_x \bv \qquad \textrm{a.e. in  } Q,\label{con26}
\end{align}
and consequently we have that $\overline{\bS}=\bS(\tilde{\psi},\bD_x \bv)$. Finally, in view of the convergence
results and a~priori estimates obtained, one can let $n\to \infty$ in
\eqref{eq:psietawf}, \eqref{eq:phietawf}
to deduce \eqref{eq:psietawf_2}, \eqref{eq:phietawf_2}. We note in this respect that, thanks to the nonnegativity of both $\psi$ (and $\phi$) the $_+$ symbol, indicating the nonnegative part of a function, can be omitted from \eqref{eq:psietawf}. Furthermore \eqref{eq:psietawf} is strongly nonlinear because of the presence of the term $\delta |\psi|^{2p-2}\psi$, but we are considering the limit $\delta\rightarrow 0_+$, so this term vanishes in the limit of $\delta \rightarrow 0_+$ thanks to the a priori estimate \eqref{6aprgood}.

Finally, by using a standard parabolic regularity result,  one can show the attainment of the initial datum \eqref{nabyvat} for $\phi$, and also that $\psi \in \mathcal{C}_{w}(0,T; L^2(\Omega_{\infty}))$ fulfils, for $t\to 0_+$,
\begin{equation}
\psi(t)\rightharpoonup \psi_0 \qquad \textrm{weakly in } L^2(\Omega_{\infty}).\label{weakinit}
\end{equation}
To strengthen this convergence result, we recall \eqref{prI2}, which, thanks to Gronwall's lemma, leads to
\begin{equation}\label{prI2used}
\begin{split}
&\int_{\Omega_{r^n_{\infty}}}\alpha(r)(\psi^n(s))^2\dd r\dd x \le {\rm e}^{C(q,\phi_0,K)t}\int_{\Omega_{r^n_{\infty}}}\alpha(r)\psi_0^2\dd r\dd x.
\end{split}
\end{equation}
Hence, letting $n\to \infty$ and using the weak lower semicontinuity of the norm function, we get
\begin{equation*}
\begin{split}
&\int_{\Omega_{{\infty}}}\alpha(r)(\psi(s))^2\dd r\dd x \le {\rm e}^{C(q,\phi_0,K)t}\int_{\Omega_{{\infty}}}\alpha(r)\psi_0^2\dd r\dd x.
\end{split}
\end{equation*}
This directly leads to
\begin{equation*}
\begin{split}
&\limsup_{t\to 0_+} \int_{\Omega_{{\infty}}}\alpha(r)(\psi(s))^2\dd r\dd x \le \int_{\Omega_{{\infty}}}\alpha(r)\psi_0^2\dd r\dd x,
\end{split}
\end{equation*}
which, when combined with \eqref{weakinit}, yields \eqref{nabyvat} for $\psi$.
\end{proof}

\subsection{The limit $\varepsilon \to 0_+$}
In this final subsection, we complete the proof of the main theorem in the paper. For this purpose, we use the existence result from Lemma~\ref{Ex2}. Hence, for $\psi_0 \in L^1(\Omega; L^1_{\theta^*_1}(r_0,\infty))\cap L^2(\Omega; L^2_{\theta_2^*}(r_0,\infty))$ with $\theta_1^*> \theta\ge 1$ and $\theta_2^*\ge 3$ we find a sequence $\psi^{\varepsilon}_0 \in \mathcal{D}(\Omega_{\infty})$ that converges strongly to $\psi_0$ in the corresponding spaces. We then denote by $(\bv^{\varepsilon}, \phi^{\varepsilon},\psi^{\varepsilon})$ the solution constructed in Lemma~\ref{Ex2} with the initial data $\phi_0$ and $\psi_0^{\varepsilon}$. Our goal is now to let $\varepsilon \to 0_+$ and obtain a solution whose existence is claimed in Theorem~\ref{T:theorem}. Henceforth, we denote by $C$ a generic constant that may depend only on the data but not on $\varepsilon$. Recalling the estimates established in the previous section, we have that \eqref{mp2} is valid and
\begin{align}
&\mbox{{\rm ess.sup}}_{t\in (0,T)} \|\bv^{\varepsilon}(t)\|_2 + \int_0^T \|\bv^{\varepsilon}\|_{1,p}^p + \|\bS^{\varepsilon}\|_{p'}^{p'} + \|\bv^{\varepsilon}\|_{L^2(\partial \Omega)}^2 +\varepsilon\|\bv^{\varepsilon}\|_{p'}^{p'} \dd t \le C,\label{5aprgood_f}\\
&\|\phi^{\varepsilon}\|_{L^{\infty}(Q)}\le C,\label{3aprgood_f}\\
&\mbox{{\rm ess.sup}}_{t\in (0,T)}\int_{\Omega_{{\infty}}}r^{\theta_2^*}(\psi^{\varepsilon})^2(t)\dd r\dd x+\int_0^T \int_{\Omega_{{\infty}}} \!\!\!\!\!\!\!\! A(r)r^{\theta_2^*}|\nabla_x \psi^{\varepsilon}|^2\dd r\dd x\dd t  \le C,
 \label{6aprgood_f}\\
&\mbox{{\rm ess.sup}}_{t \in (0,T)} \int_{\Omega_{{\infty}}}r^{\theta_1^*}\psi^{\varepsilon}(t)\dd r\dd x \le  C.
\label{7aprgood_f}
\end{align}

\noindent
Therefore, using the same arguments as before, one can deduce that there exists a quadruple $(\bv, \overline{\bS}, \psi,\phi)$ such that, for sufficiently large $q>1$, one has
\begin{align}
\bv^{\varepsilon} &\rightharpoonup^* \bv &&\textrm{weakly$^*$ in } L^p(0,T; W^{1,p}_{\bn, \diver})\cap L^{\infty}(0,T; L^2(\Omega)^d),\label{con1_f}\\
\bS^{\varepsilon} &\rightharpoonup \overline{\bS} &&\textrm{weakly in } L^{p'}(Q)^{d\times d},\label{con2_f}\\
\phi^{\varepsilon} &\rightharpoonup^* \phi &&\textrm{weakly$^*$ in } L^{\infty}(Q),\label{con3_f}\\
\phi^{\varepsilon} &\rightharpoonup \phi &&\textrm{weakly in } L^2(0,T; W^{1,2}(\Omega)),\label{con4_f}\\
\psi^{\varepsilon} &\rightharpoonup \psi &&\textrm{weakly in } L^2(Q_{\infty})\cap L^2(0,T; L^2_{loc}(r_0,\infty; W^{1,2}(\Omega))),\label{con5_f}\\
\psi^{\varepsilon} &\rightharpoonup^* \psi &&\textrm{weakly$^*$ in } L^{\infty}(0,T; L^1(\Omega; L^1_{\theta^*_1}(r_0,\infty)))\cap L^{\infty}(0,T; L^2(\Omega; L^2_{\theta^*_2}(r_0,\infty))),\label{con5.1_f}\\
\bv^{\varepsilon} &\rightharpoonup \bv &&\textrm{weakly in } L^2(0,T; L^2(\partial \Omega)^d),\label{con5.5_f}\\
\partial_t \bv^{\varepsilon} &\rightharpoonup \partial_t \bv &&\textrm{weakly in } L^q(0,T; W^{-1,q'}_{\bn, \diver}),\label{con6_f}\\
\partial_t \phi^{\varepsilon} &\rightharpoonup \partial_t\phi &&\textrm{weakly in } (L^2(0,T; W^{1,2}(\Omega)))^*,\label{con7_f}\\
\partial_t \psi^{\varepsilon} &\rightharpoonup \partial_t\psi &&\textrm{weakly in } (L^q(0,T; W^{1,2}(\Omega_{\infty})\cap W^{1,q}(\Omega_{\infty})))^*,\label{con8_f}\\
\bv^n &\to \bv &&\textrm{a.e. in  } Q,\label{con9_f}\\
\phi^n &\to  \phi &&\textrm{a.e. in  } Q.\label{con10_f}
\end{align}
Based on these convergence results, we can deduce similarly as before that, for a sequence of $\gamma_{\delta} \in \mathcal{D}(r_0,\infty)$ such that $\gamma_{\delta}\nearrow \gamma$, almost everywhere and also locally in $\mathcal{C}(r_0,\infty)$, we have
\begin{align}
\int_{r_0}^{\infty}\gamma_{\delta}(r)\psi^{\varepsilon}(t,x,r)\dd r &\to \int_{r_0}^{\infty}\gamma_{\delta}(r)\psi(t,x,r)\dd r \qquad\textrm{strongly in  } L^2(Q).\label{con18_f}
\end{align}
Next, we slightly change the convergence result for the original sequence since the initial data are assumed to be only in $L^1_{\theta_1^*}$. Thus, we focus only on $L^1$ convergence; in particular, we note that
\begin{equation}\label{original_f}
\begin{split}
\int_Q |\tilde{\psi}^{\varepsilon} -\tilde{\psi}|\dd x\dd t&:= \int_Q \left|\int_{r_0}^{\infty}\gamma(r)(\psi^{\varepsilon}(r,t,x)-\psi(r,t,x))\dd r \right|\dd x\dd t\\
&\le \int_Q \left|\int_{r_0}^{\infty}(\gamma(r)-\gamma_{\delta}(r))(\psi^{\varepsilon}(r,t,x)-\psi(r,t,x))\dd r \right|\dd x\dd t\\
&\quad +\int_Q \left|\int_{r_0}^{\infty}\gamma_{\delta}(r)(\psi^{\varepsilon}(r,t,x)-\psi(r,t,x))\dd r \right|\dd x\dd t.
\end{split}
\end{equation}
Consequently, using \eqref{con18_f}, we see that
\begin{equation}\label{original2_f}
\begin{split}
\lim_{\varepsilon\to 0_+}\int_Q |\tilde{\psi}^{\varepsilon} -\tilde{\psi}|\dd x\dd t&\le \lim_{\varepsilon \to 0_+}\int_Q \left|\int_{r_0}^{r^*}(\gamma(r)-\gamma_{\delta}(r))(\psi^{\varepsilon}(r,t,x)-\psi(r,t,x))\dd r \right|\dd x\dd t\\
&\quad  +\lim_{\varepsilon\to 0_+}\int_Q \left|\int_{r^*}^{\infty}|\gamma(r)||\psi^{\varepsilon}(r,t,x)-\psi(r,t,x)|\dd r \right|\dd x\dd t,
\end{split}
\end{equation}
and thanks to  the uniform convergence of $\gamma_{\delta}$ on compact sets, we can also easily let $\delta \to 0_+$ (with the help of the uniform bound \eqref{7aprgood_f}) to deduce using \eqref{growthgamma} that
\begin{equation}\label{original3_f}
\begin{split}
&\lim_{\varepsilon\to 0_+}\int_Q |\tilde{\psi}^{\varepsilon} -\tilde{\psi}|\dd x\dd t\le \lim_{\varepsilon \to 0_+}\int_Q \left|\int_{r^*}^{\infty}|\gamma(r)||\psi^{\varepsilon}(r,t,x)-\psi(r,t,x)|\dd r \right|\dd x\dd t\\
&\le C\lim_{\varepsilon \to 0_+}\int_Q \left|\int_{r^*}^{\infty}r^{\theta}|\psi^{\varepsilon}(r,t,x)-\psi(r,t,x)|\dd r \right|\dd x\dd t\\
&\le C(r^*)^{\theta - \theta_1^*}(\|\psi^{\varepsilon}\|_{L^1_{\theta_1^*}}+\|\psi^{\varepsilon}\|_{L^1_{\theta_1^*}})\le C(r^*)^{\theta - \theta_1^*}\overset{r^*\to \infty}{\to} 0,
\end{split}
\end{equation}
and consequently
\begin{align}
\tilde{\psi}^{\varepsilon} &\to \tilde{\psi}\qquad\textrm{strongly in  } L^1(Q).\label{con20_f}
\end{align}
Based on these convergence results, it is now an easy task to let $\varepsilon \to 0_+$ in \eqref{NSeps-wf_2}--\eqref{eq:phietawf_2} to get \eqref{T:3}--\eqref{T:3b} provided that we can show that
\begin{align}
\bD_x \bv^{\varepsilon} \to \bD_x \bv \qquad\textrm{strongly in  } L^1(Q)^{d\times d}.\label{con21_f}
\end{align}
In addition, the proof of the attainment of the initial data for $\bv$ and $\phi$ is rather standard, and for $\psi$ we can use the same scheme as in the previous section. Also to prove the conservation of mass identity for the polymer chains, we follow the proof of Lemma~\ref{L-mass}, but with a proper cut-off function. Specifically, we can set
$$\mbox{$z:=\chi_{[0,s]}$ in \eqref{T:3b} and $\varphi:=\chi_{[0,s]} r \eta_k(r)$ in \eqref{T:3a}},$$
with arbitrary $\eta_k \in \mathcal{D}(\mathbb{R})$, and after summing the resulting identities we deduce that, for almost all $t\in (0,t)$, we have
\begin{equation}\label{psi-cal3_f}
\begin{split}
&\int_\Omega \phi(s,x)-\phi_0(x) +\left(\int_{r_0}^{\infty} r\eta_k(\psi(s,x,r)-\psi_0(x,r)) \dd r\right) \dd x\\
&= -\int_0^s\int_{Q_{\infty}}r\eta_k \beta(r, \cdot)\psi(t,x,r) \dd x \dd r\dd t\\
&\qquad + 2\int_0^s\int_{\Omega_{\infty}} r\eta_k\int_r^{\infty}\kappa(r,\tilde r)\beta(\tilde r,\cdot)\psi(t,x,\tilde r)\dd\tilde r  \dd x \dd r\dd t\\
&\qquad + \int_0^s\int_{\Omega_{\infty}} (\partial_r(r\tau \eta_k)-\partial_r(r\tau))  \phi(t,x) \psi(t,x,r) \dd x \dd r\dd t\\
&\qquad + 2\int_0^s\int_{\Omega} \int_0^{r_0} r \int_{r_0}^{\infty}\beta(\tilde r,\cdot)\kappa(r,\tilde r)\psi(t,x,\tilde r)\dd\tilde r \dd r\dd x\dd t.
\end{split}
\end{equation}
Next, we assume that $\eta_k \equiv 1$ on $(0,r_0)$ and we follow the proof of Lemma~\ref{L-mass} to get, with the help of \eqref{df-kappa}, the following chain of equalities:
$$
\begin{aligned}
&2\int_0^s\int_{\Omega_{\infty}} r\eta_k\int_r^{\infty}\kappa(r,\tilde r)\beta(\tilde r,\cdot)\psi(t,x,\tilde r)\dd\tilde r  \dd x \dd r\dd t\\
&\quad +2\int_0^s\int_{\Omega} \int_0^{r_0} r \int_{r_0}^{\infty}\beta(\tilde r,\cdot)\kappa(r,\tilde r)\psi(t,x,\tilde r)\dd\tilde r \dd r\dd x\dd t\\
&=2\int_0^s\int_{\Omega_{\infty}} r\eta_k\int_r^{\infty}\kappa(r,\tilde r)\beta(\tilde r,\cdot)\psi(t,x,\tilde r)\dd\tilde r  \dd x \dd r\dd t\\
&\quad +2\int_0^s\int_{\Omega} \int_0^{r_0} r\eta_k \int_{r}^{\infty}\beta(\tilde r,\cdot)\kappa(r,\tilde r)\psi(t,x,\tilde r)\dd\tilde r \dd r\dd x\dd t\\
&=2\int_0^s\int_{\Omega} \int_0^{\infty}r\eta_k(r)\int_r^{\infty}\kappa(r,\tilde r)\beta(\tilde r,\cdot)\psi(t,x,\tilde r)\dd\tilde r  \dd x \dd r\dd t\\
&=2\int_0^s\int_{\Omega} \int_0^{\infty}\left(\int_0^{\tilde{r}} r\eta_k(r)\kappa(r,\tilde r)\dd r\right) \beta(\tilde r,\cdot)\psi(t,x,\tilde r)\dd\tilde r  \dd x\dd t\\
&=2\int_0^s\int_{\Omega} \int_{r_0}^{\infty}\left(\int_0^{\infty} r\eta_k(r)\kappa(r,\tilde r)\dd r\right) \beta(\tilde r,\cdot)\psi(t,x,\tilde r)\dd\tilde r  \dd x\dd t\\
&=\int_0^s\int_{\Omega} \int_{r_0}^{\infty}\tilde{r}\eta_k(\tilde{r})\beta(\tilde r,\cdot)\psi(t,x,\tilde r)\dd\tilde r  \dd x\dd t\\
&\quad -\int_0^s\int_{\Omega} \int_{r_0}^{\infty} \tilde{r}^{-1}\left(\int_0^{\tilde{r}} r^2\eta'_k(r)\dd r\right) \beta(\tilde r,\cdot)\psi(t,x,\tilde r)\dd\tilde r  \dd x\dd t.
\end{aligned}
$$
By substituting this identity into \eqref{psi-cal3_f}, we obtain
\begin{equation}\label{psi-cal3_ff}
\begin{split}
&\int_\Omega \phi(s,x)-\phi_0(x) +\left(\int_{r_0}^{\infty} r\eta_k(\psi(s,x,r)-\psi_0(x,r)) \dd r\right) \dd x\\
&= \int_0^s\int_{\Omega_{\infty}} (r\tau \eta'_k(r) + (\eta_k(r)-1) \partial_r(r\tau))  \phi(t,x) \psi(t,x,r) \dd x \dd r\dd t\\
&\quad -\int_0^s\int_{\Omega} \int_{r_0}^{\infty} \tilde{r}^{-1}\left(\int_0^{\tilde{r}} r^2\eta'_k(r)\dd r\right) \beta(\tilde r,\cdot)\psi(t,x,\tilde r)\dd\tilde r  \dd x\dd t.
\end{split}
\end{equation}
Next, we let $\eta_k \nearrow 1$. Thus, we set $\eta_k(r)\equiv 1$ for $r\le k$ and $\eta_k(r) \equiv 0$ for $r\ge 2k$ such that $|\eta'_k|\le Ck^{-1}$. Using this definition it is not difficult to let $k \to \infty$ in the terms on the left-hand side of \eqref{psi-cal3_ff} by using the monotone convergence theorem. For the term on the right-hand side, we can use \eqref{tau-a} to deduce that
$$
\begin{aligned}
&\left|\int_0^s\int_{\Omega_{\infty}} (r\tau \eta'_k(r) + (\eta_k(r)-1) \partial_r(r\tau))  \phi(t,x) \psi(t,x,r) \dd x \dd r\dd t\right.\\
&\quad \left.-\int_0^s\int_{\Omega} \int_{r_0}^{\infty}\tilde{r}^{-1}\left(\int_0^{\tilde{r}} r^2\eta'_k(r)\dd r\right) \beta(\tilde r,\cdot)\psi(t,x,\tilde r)\dd\tilde r  \dd x\dd t\right|\\
&\le C\int_0^s\int_{\Omega}\int_{k}^{\infty}\phi(t,x) \psi(t,x,r) \dd x \dd r\dd t+ C\int_0^s\int_{\Omega} \int_{k}^{\infty}\tilde{r}\psi(t,x,\tilde r)\dd\tilde r  \dd x\dd t\\
&\le C(\|\psi\|_{\infty})\int_0^s\int_{\Omega} \int_{k}^{\infty}\tilde{r}\psi(t,x,\tilde r)\dd\tilde r  \dd x\dd t \overset{k\to \infty}\to 0,
\end{aligned}
$$
where the last convergence follows from the fact that $\psi \in L^1_1(Q_{\infty})$.

Hence, it remains to show \eqref{con21_f}. First, following (3.56), (3.57) and (3.60) in \cite{BuGwMaSw2012}, using the fact that $\Omega \in \mathcal{C}^{1,1}$ we can find $q^{\varepsilon}_1$, $q^{\varepsilon}_2$ and some $q^*>1$ such that
\begin{alignat}{2}
q^{\varepsilon}_1 &\rightharpoonup q_1 &&\qquad\textrm{ weakly in } L^{p'}(Q), \label{pr1}\\
q^{\varepsilon}_2 &\rightharpoonup q_2 &&\qquad \textrm{ weakly in }L^{q*}(Q), \label{pr1.5}\\
q^{\varepsilon}_2 &\to q_2 &&\qquad\textrm{ strongly in } L^{h}(Q) \textrm{ for all $h \in [1,q^*)$}, \label{pr2}
\end{alignat}
fulfilling, for all $\bw \in L^{\infty}(0,T; W^{1,\infty}(\Omega)^d \cap W^{1,1}_{\bn})$,
\begin{equation}\label{NSeps-tlak}
\begin{split}
&\int_0^T\langle\partial_t\bv^{\varepsilon},\bw \rangle -\int_Q \bv^{\varepsilon}\otimes \bv^{\varepsilon} \cdot \nabla_x \bw +\bS(\tilde{\psi}^{\varepsilon},\bD_x \bv^{\varepsilon})\cdot \nabla_x \bw+\varepsilon |\bv^{\varepsilon}|^{2p'-2}\bv^{\varepsilon}\cdot \bw \dd x\dd t\\
&= -\alpha^*\int_0^T \int_{\partial \Omega}\bv^{\varepsilon} \cdot \bw\dd S +\int_0^T\langle \bff,\bw\rangle \dd t+\int_Q (q_1^{\varepsilon}+q_2^{\varepsilon})\diver_{x} \bw \dd x\dd t,
\end{split}
\end{equation}
and using the convergence results \eqref{con1_f}--\eqref{con10_f}, we also get
\begin{equation}\label{NS-tlak}
\begin{split}
&\int_0^T\langle\partial_t\bv,\bw \rangle -\int_Q \bv\otimes \bv \cdot \nabla_x \bw +\overline{\bS}\cdot \nabla_x \bw\dd x\dd t\\
&= -\alpha^*\int_0^T \int_{\partial \Omega}\bv\cdot \bw\dd S +\int_0^T\langle \bff,\bw\rangle \dd t+\int_Q (q_1+q_2)\diver_{x} \bw \dd x\dd t,
\end{split}
\end{equation}
which is nothing else than \eqref{T:3} with $q$ given as $q:=q_1+q_2$ provided we show \eqref{con21_f} to identify $\overline{\bS}$.
We now set $n:=[\varepsilon^{-1}]$, reinstate the index $n$ for all functions concerned, and define
$$
\begin{aligned}
\bu^n&:= \bv^n - \bv,\\
\bef^n&:= -\frac{1}{n}|\bv^n|^{2p'-2}\bv^n,\\
\bG^n&:= \bv^n \otimes \bv^n - \bv \otimes \bv + (q^n_2 -q_2)\bI,\\
\bH^n&:= -\bS(\tilde{\psi}^n, \bD_x \bv^n)+q_1^n \bI,\\
\bH&:= -\overline{\bS} +q_1 \bI.
\end{aligned}
$$
Hence, we may apply Lemma~\ref{Lip-Lem}. Thus, we set $\lambda^*:=k$ and we see that we can find a sequence $\lambda_k^n\in [k,Ck^{p^k}]$ and the corresponding sequence $\bu^{n,k}$ fulfilling \eqref{strong}, \eqref{norm}. Next, for any nonnegative $g\in \mathcal{D}(Q)$ we set $\bw:=\bu^{n,k}$ in \eqref{NSeps-tlak}, \eqref{NS-tlak}; by subtracting the resulting equations, using \eqref{bG}, \eqref{bf} and \eqref{tdercon}, we obtain
\begin{equation}\label{dream}
\limsup_{n\to \infty} \int_{Q} (\bS(\tilde{\psi}^n,\bD_x \bv^n) - \overline{\bS}) \cdot \bD_x (\bu^{n,k})g - g(q_1^n-q_1)\diver_x \bu^{n,k}\dd x\dd t \le C(g) (k^{1-p} + k^{-1})^{\beta}.
\end{equation}
Moreover, by using \eqref{strong}, \eqref{norm}, and also the fact that
$$
\bS(\tilde{\psi}^n,\bD_x \bv) \to \bS(\tilde{\psi},\bD_x \bv) \quad \textrm{strongly in } L^{p'}(Q)^{d\times d},
$$
which is the consequence of \eqref{con20_f}, Lebesgue's dominated convergence theorem and the assumption \eqref{S},
we see that \eqref{dream} reduces to
\begin{equation}\label{dream2}
\limsup_{n\to \infty} \int_{Q} (\bS(\tilde{\psi}^n,\bD_x \bv^n) - \bS(\tilde{\psi}^n,\bD_x \bv)) \cdot \bD_x (\bu^{n,k})g - gq_1^n\diver_x \bu^{n,k}\dd x\dd t \le C(g) (k^{1-p} + k^{-1})^{\beta},
\end{equation}
and consequently, after denoting by $Q_g$ the support of $g$ and using the definition of $E^n_k$, we have that
\begin{equation}\label{dream3}
\begin{split}
&\limsup_{n\to \infty} \int_{Q_g\setminus E^n_k} (\bS(\tilde{\psi}^n,\bD_x \bv^n) - \bS(\tilde{\psi}^n,\bD_x \bv)) \cdot \bD_x (\bv^n-\bv)g\dd x\dd t \\
&\le C(g)\limsup_{n\to \infty} \int_{Q_g \cap E^n_k} (|\bH^n|+|\bH|)|\bD_x (\bu^{n,k})| \dd x\dd t+ C(g) (k^{1-p} + k^{-1})^{\beta}\\
&\le C(g)(k^{1-p}+k^{-\beta})+  C(g)(k^{1-p} + k^{-1})^{\beta}.
\end{split}
\end{equation}
Finally, using the monotonicity of $\bS$, H\"{o}lder's inequality and \eqref{meas}, we have that
\begin{equation}\label{dream4}
\begin{split}
&\limsup_{n\to \infty} \int_{Q} \sqrt{(\bS(\tilde{\psi}^n,\bD_x \bv^n) - \bS(\tilde{\psi}^n,\bD_x \bv)) \cdot \bD_x (\bv^n-\bv)}g\dd x\dd t\\
&\le \limsup_{n\to \infty} \int_{Q_g\setminus E^n_k} \sqrt{(\bS(\tilde{\psi}^n,\bD_x \bv^n) - \bS(\tilde{\psi}^n,\bD_x \bv)) \cdot \bD_x (\bv^n-\bv)}g\dd x\dd t\\
&\quad +\limsup_{n\to \infty} \int_{Q_g\cap E^n_k} \sqrt{(\bS(\tilde{\psi}^n,\bD_x \bv^n) - \bS(\tilde{\psi}^n,\bD_x \bv)) \cdot \bD_x (\bv^n-\bv)}g\dd x\dd t\\
&\le C\limsup_{n\to \infty} \left(\int_{Q_g\setminus E^n_k} (\bS(\tilde{\psi}^n,\bD_x \bv^n) - \bS(\tilde{\psi}^n,\bD_x \bv)) \cdot \bD_x (\bv^n-\bv)g\dd x\dd t\right)^{\frac12} + C|Q_g \cap E^n_k|^{\frac12} \\
&\le C(g)(k^{1-p}+k^{-\beta})^{\frac12}+  C(g)(k^{1-p} + k^{-1})^{\frac{\beta}{2}}+Ck^{-\frac{p}{2}} \overset{k\to \infty}\to 0.
\end{split}
\end{equation}
Thus, using the strict monotonicity of $\bS$, see \eqref{S}, we deduce (for a subsequence) \eqref{con21_f}, which completes the proof.  \hfill $\Box$

\bigskip

\begin{Remarkno}In order to keep the length of the paper within reason, we focused here
on the mathematical analysis of a model that only admits polymerization
between a polymer and a monomer. The mathematical analysis of a more
complicated model, which also includes polymerization between different
polymer chains, is of interest. It would involve coupling the
generalized Navier--Stokes system of a form considered here, with a
continuous coagulation-fragmentation equation, where the viscosity
coefficient in the Navier--Stokes momentum equation depends on the size
distribution function satisfying a coagulation-fragmentation equation
whose nonlocal terms account for the formation of polymer chains by
coalescence of smaller chains, the breakage of polymer chains into
two smaller pieces, the depletion of polymer chains by coagulation with
other polymer chains, and the gain of polymer chains as a result of the fragmentation of larger chains. We expect though that the analysis
of such a more complicated model would proceed along similar lines.
\end{Remarkno}

\bibliographystyle{abbrv}
\bibliography{lipschitz}
\end{document}